\RequirePackage{cmap}
\documentclass[12pt,a4paper]{article}
\usepackage[T1]{fontenc}
\usepackage{graphicx}
\graphicspath{{images/}}
\DeclareGraphicsExtensions{.png}
\usepackage{floatflt}
\usepackage{amsfonts}                              
\usepackage{amsthm}                                
\usepackage{amsmath}
\usepackage{gensymb}
\usepackage{amssymb,amsfonts}
\usepackage[usenames]{color}
\usepackage{wrapfig}
\usepackage{url}
\usepackage{cmap}
\usepackage{graphics}
\usepackage{multirow}
\usepackage{comment}
\usepackage{caption}
\usepackage{subcaption}
\usepackage{float}
\renewcommand{\emptyset}{\varnothing}

\makeatletter
\renewcommand\section{\@startsection {section}{1}{\z@}%
                                   {-3.5ex \@plus -1ex \@minus -.2ex}%
                                   {2.3ex \@plus.2ex}%
                                   {\normalfont\LARGE\bfseries}}
\renewcommand\subsection{\@startsection{subsection}{2}{\z@}%
                                     {-3.25ex\@plus -1ex \@minus -.2ex}%
                                     {1.5ex \@plus .2ex}%
                                     {\normalfont\Large\bfseries}}
\renewcommand\subsubsection{\@startsection{subsubsection}{3}{\z@}%
                                     {-3.25ex\@plus -1ex \@minus -.2ex}%
                                     {1.5ex \@plus .2ex}%
                                     {\normalfont\large\bfseries}}
\makeatother

\usepackage[left=1.65cm,right=1.65cm,
    top=1.57cm,bottom=1.57cm,bindingoffset=0cm]{geometry}

\usepackage{amsthm}

\theoremstyle{definition}
\newtheorem{remark}{Remark}
\newtheorem{conjecture}{Conjecture}
\newtheorem{example}{Example}
\newtheorem{definition}{Definition}
\newtheorem{Symbol}{Notation}
\newtheorem{theorem}{Theorem}
\newtheorem{lemma}{Lemma}
\newtheorem{corollary}{Corollary}
\newtheorem*{agreement}{Agreement}
\providecommand{\keywords}[1]
{
  \textbf{\textit{Keywords---}} #1
}

\author{I.\,V.~Novikov}
\title{\textbf{Percolation of three fluids on a honeycomb lattice}}
\date{}

\begin{document}

\setcounter{section}{0}

\maketitle
\setcounter{section}{0}
\begin{abstract}
    In this paper, we consider a generalization of percolation: percolation of three related fluids on a honeycomb lattice. K.~Izyurov and A.~Magazinov proved that percolations of distinct fluids between opposite sides on a fixed hexagon become mutually independent as the lattice step tends to $0$. This paper exposes this proof in details (with minor simplifications) for nonspecialists. In addition, we state a few related conjectures based on numerical experiments.
\end{abstract}


\tableofcontents

\section{Introduction}
    In this paper, we consider a generalization of percolation: percolation of three related fluids on a honeycomb lattice. We prove (see Theorem \ref{sides}) that percolations of distinct fluids between opposite sides of a fixed hexagon become mutually independent as the lattice step tends to~$0$. This was a conjecture by M. Skopenkov proved independently by K.~Izyurov and A.~Magazinov approximately at the same time (private communication). The proof is based on the Fourier-Walsh expansion and Kesten's theorem. This paper exposes this proof in details (with minor simplifications) for nonspecialists. We also state a few new related conjectures based on numerical experiments.

    The paper is organized as follows. In \S2, we introduce key definitions, and also state Main Theorem \ref{sides} and several conjectures. In \S3, we introduce the Fourier-Walsh expansion. In \S4, we expose the proof of Theorem \ref{sides} based on the Fourier-Walsh expansion. In \S5, we describe results of numerical experiments related to conjectures from \S2.
\section{Main theorem}
     For any $n>1$ consider a honeycomb lattice of step $\frac{1}{n}$. Let $P$ be a regular hexagon with the side~$1$ centered at the center of some cell~$O$. Denote by $M_n$ the set of all the cells contained in $P$. Consider the probability space~$\Omega$ consisting of all the colorings of the cells of the set~$M_n$ into $4$ colors, denoted by $0$, $1$, $2$, $3$, with the measure $P(B)=|B|/4^{|M_n|}$ for any $B \subset \Omega$. The probability space $\Omega$ is called \textit{the four-state Potts model at infinite temperature.}

    \begin{definition}
    \label{4 color sides}
        Fix a number $k=1$, $2$ or $3$, and also a coloring of the cells of the set~$M_n$ into $4$ colors. We say that \textit{fluid $k$ percolates between two sets $A, B \subset~M_n$}, if some cell of the set $A$ is joined with some cell of the set $B$ by a chain of adjacent cells such that each cell in the chain has color $0$ or $k$, including the initial cell of the set $A$ and the final cell of the set $B$.
    \end{definition}
\subsection{Percolation between sides}
    We say that \textit{a cell $x \in M_n$ belongs to a side $A$ of the hexagon $P$}, if the cell $x$ is a boundary cell and $A$ is the nearest side to $x$. A cell can belong to more than one side (if the cell has several nearest sides). In what follows, by \textit{side $A$} we mean the set of all the cells belonging to the side $A$ of the hexagon~$P$.

    Label the sides of the hexagon $P$ by numbers from $1$ to $6$ counterclock-wise. For $k=1, 2$ or $3$ denote by $B_{k,n} \subset \Omega$ the set of all the colorings such that fluid $k$ percolates between sides $k$ and $k+3$ of the hexagon~$P$. It is easy to show that events $B_{1,n}, B_{2,n}, B_{3,n}$ are pairwise independent: $P(B_{1,n} \cap B_{2,n}) - P(B_{1,n})P(B_{2,n})=0$.

    \begin{theorem}[K.~Izyurov, A.~Magazinov, 2018]
    \label{sides}
        $$\displaystyle{\lim_{n \to +\infty}} \left[P(B_{1,n} \cap B_{2,n} \cap B_{3,n}) - P(B_{1,n}) P(B_{2,n}) P(B_{3,n})\right] = 0.$$
    \end{theorem}

    \begin{remark}
    Theorem \ref{sides} holds in a much more general situation. For example, a regular hexagon~$P$ can be replaced by an arbitrary polygon, and opposite sides can be replaced by arbitrary pairs of sides. Nevertheless, for simplicity of the proof, we consider percolation between opposite sides of a regular hexagon.
    \end{remark}

    Informally, Theorem \ref{sides} states that percolations of distinct fluids between opposite sides become mutually independent as $n$ tends to~$\infty$. We take a difference of probabilities rather than a ratio to avoid proving that $P(B_{1,n}) P(B_{2,n}) P(B_{3,n})$ is bounded from zero.

    In addition, we state the following conjecture.
    \begin{conjecture}
    \label{hyp3}
        $P(B_{1,n} \cap B_{2,n} \cap B_{3,n}) \geqslant P(B_{1,n}) P(B_{2,n}) P(B_{3,n})$ for each $n$.
    \end{conjecture}

    Since events $B_{1,n}, B_{2,n}, B_{3,n}$ are independent, the conjecture states that percolations of distinct fluids are positively correlated, i.e., $$P(B_{1,n}|B_{2,n} \cap B_{3,n}) \geqslant B_{1,n}.$$

\subsection{Percolation from the center}
    \begin{definition}
    \label{4 color center}
        Fix a number $k=1$, $2$ or $3$, and also a coloring of the cells of the set $M_n$ into $4$ colors. We say that \textit{fluid $k$ percolates from a cell $x \in M_n$ to a set $A \in M_n$}, if $x$ is joined with some cell of the set $A$ by a chain of adjacent cells such that each cell in the chain has color $0$ or $k$, including the final cell of the set $A$ and \textit{not} including the initial cell $x$.
    \end{definition}

    Note that in Definition \ref{4 color center}, unlike Definition \ref{4 color sides}, the color of the initial cell $x$ does not matter. That is convenient to avoid a factor of $2$ in the inequalities of Conjectures \ref{hyp1} and \ref{hyp2} below.

    For $k = 1, 2$, and $3$ denote by $A_{k,n} \subset \Omega$ the set of colorings such that fluid $k$ percolates from the cell $O$ to the set of all the boundary cells of the set $M_n$.

    Let us state 2 conjectures.

    \begin{conjecture}
    \label{hyp1}
        $P(A_{1,n} \cap A_{2,n} \cap A_{3,n}) \geqslant P(A_{1,n}) P(A_{2,n}) P(A_{3,n})$ for each $n$.\\
    \end{conjecture}

    \begin{conjecture}
    \label{hyp2}
        $\displaystyle{\lim_{n \to +\infty}} \frac{P(A_{1,n} \cap A_{2,n} \cap A_{3,n})}{P(A_{1,n}) P(A_{2,n}) P(A_{3,n})} > 1$.
    \end{conjecture}


    In addition, we state one more conjecture that generalizes Conjectures \ref{hyp3} and \ref{hyp1}.
    \begin{conjecture}
    \label{hyp4}
        Let $A_1, A_2, A_3, B_1, B_2, B_3 \subset M_n$. Denote by $U_k = $ \{fluid $k$ percolates between $A_k$ and~$B_k$\}. Then
        $$P(U_{1} \cap U_{2} \cap U_{3}) \geqslant P(U_{1}) P(U_{2}) P(U_{3}).$$
    \end{conjecture}
    If Conjecture \ref{hyp4} is true, then $P(U_{1}|U_{2} \cap U_{3}) \geqslant P(U_{1})$, i.e., percolations between any three pairs of subsets are positively correlated.

\subsection{Example}
    \begin{wrapfigure}{o}{0.2\linewidth} 
        \center{\includegraphics[width=1\linewidth]{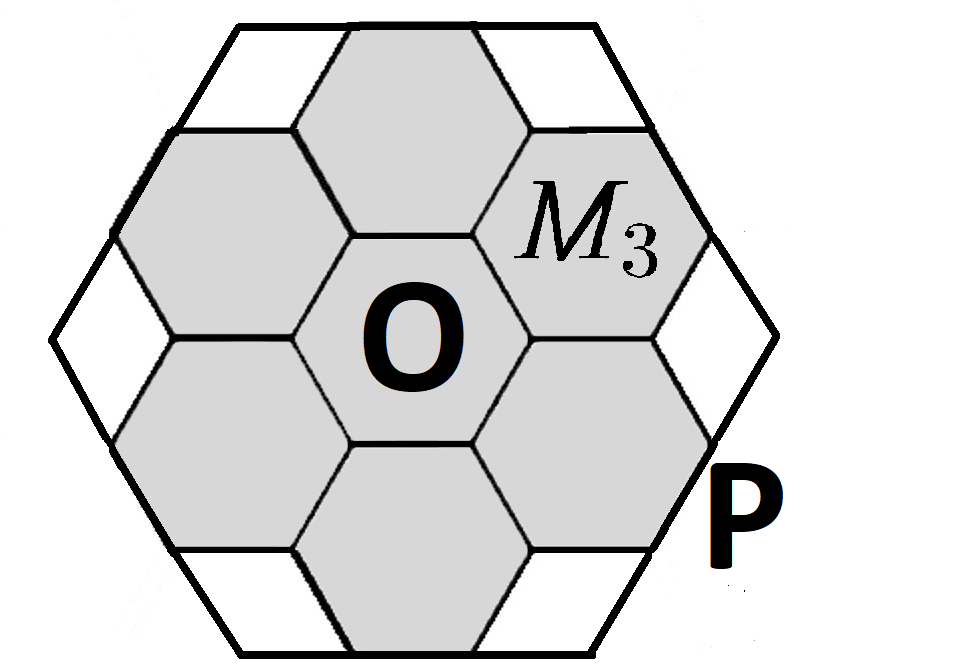}}
        \caption{}
    \end{wrapfigure}
    We illustrate the above notions by an example.

    \begin{example}
        Let $n=3$. In Figure 1 we see a regular hexagon~$P$ on a honeycomb lattice of step $\frac{1}{3}$. The set $M_3$ consists of 7 gray cells. It is easy to show that
        $$P(A_{1,3}) = P(A_{2,3}) = P(A_{3,3}) = 1 - \left(\frac{1}{2}\right)^6 = 0,984375;$$
        $$P(A_{1,3} \cap A_{2,3} \cap A_{3,3}) = 1 - 3\left(\frac{1}{2}\right)^6 + 3 \left(\frac{1}{2}\right)^{12} =0,953857421875;$$
        $$P(B_{1,3}) = P(B_{2,3}) = P(B_{3,3}) = \frac{1}{4}\left(\frac{1}{2} + \frac{1}{2} \cdot \frac{7}{16}\right)= \frac{23}{128} = 0,1796875;$$
        $$P(B_{1,3} \cap B_{2,3} \cap B_{3,3}) =
        \frac{1}{4} \cdot \left(\frac{1}{2}\right)^6 + \frac{3}{4} \cdot \left( 9 \left(\frac{1}{4}\right)^6 + 2 \left(\frac{1}{2}\right)^2 \left(\frac{1}{4}\right)^4\right) = 0,00701904296875.$$

        Then
        $$\frac{P(A_{1,3} \cap A_{2,3} \cap A_{3,3})}{P(A_{1,3}) P(A_{2,3}) P(A_{3,3})} = \frac{250048}{250047} > 1
        \text{ and }$$
        $$P(B_{1,3} \cap B_{2,3} \cap B_{3,3}) - P(B_{1,3}) P(B_{2,3}) P(B_{3,3}) \approx 0,01.$$
        This agrees with Conjectures \ref{hyp3} and \ref{hyp1}, and also show that both events $A_{1,3}, A_{2,3}, A_{3,3}$ and events $B_{1,3}, B_{2,3}, B_{3,3}$ are mutually \textit{dependent}.
    \end{example}
\subsection{Kesten's Theorem}
    The model introduced above naturally generalizes the classical percolation model (when cells are paint into 2 colors). In that model, the following famous result, used in the proof of Theorem \ref{sides}, holds.

    \begin{definition}[cf. Definition \ref{4 color center}]
    \label{2 color center}
        Paint cells of some finite set $M$ into $2$ colors $\pm 1$. We say that \textit{there exists percolation from a cell $x \in M$ to a set $A \subset M$} for that coloring, if $x$ is joined with some cell of the set $A$ by a chain of adjacent cells such that each cell in the chain has color $+1$, including the final cell of the set $A$ and \textit{not} including the initial cell $x$. We introduce the measure $P(B)~=~|B|/2^m$ on the set of all the colorings of the set $M$ into 2 colors.
    \end{definition}

    \begin{theorem}(H. Kesten, 1982, cf. \cite[Theorem~9.6]{Ariel Yadin}).
        \label{kesten lite} The probability that there exists percolation from the cell $O$ to the boundary of the set $M$ tends to $0$ as $n\to\infty$.
    \end{theorem}

\section{The Fourier-Walsh expansion}
    \label{section fourier}
        In this section, we introduce the notion of the Fourier-Walsh expansion and consider some properties of the coefficients of that expansion. The content of \S\S \ref{Fourier subsection 1}-\ref{Fourier subsection 3} is borrowed from \cite[\S\S 1.2-1.4]{boolean functions}. Lemma~\ref{formula probability} from \S 3.4 is similar to a result from \cite[\S 2.2]{boolean functions}. Lemma~\ref{inequality} from \S 3.4 is contained in \cite[\S 3.6]{boolean functions} as an exercise.

    \begin{agreement}
        In what follows, on the set $\{-1,+1\}^m$, where $m~\in~\mathbb{N}$, fix the measure $P(B) = |B|/2^m$ for any $B \subset \{-1,+1\}^m$. Functions $f \colon \{-1,+1\}^m \to \mathbb{R}$ are considered as random variables on $\{-1,+1\}^m$.
    \end{agreement}

\subsection{Definition}
\label{Fourier subsection 1}

    \begin{definition}
    \label{Fourier}
        Let $m \in \mathbb{N}$. The \textit{Fourier-Walsh expansion of a function $$f \colon \{-1,+1\}^m \to \mathbb{R}$$} is its representation as the sum $$f(x_1, \ldots, x_m) = \sum_{S \subset \{1, 2, \ldots, m\}} \widehat{f_{S}} \prod_{i \in S} x_i,$$ where $\widehat{f_{S}}$ are some real numbers that are called the \textit{ Fourier-Walsh coefficients}. If $S = \emptyset$, then we put by definition $\prod_{i \in S} x_i = 1$.
    \end{definition}

    \begin{remark}
         In what follows, we write $\widehat{f_{S}}$ for the coefficient of a term $\prod_{i \in S} x_i$ in the Fourier-Walsh expansion of a function $f \colon \{-1,+1\}^m \to \mathbb{R}$.
    \end{remark}

    \begin{theorem}
        For any $f \colon \{-1,+1\}^m \to \mathbb{R}$, there exists a unique Fourier-Walsh expansion.
    \end{theorem}
    \begin{proof} $ $\newline
        \textit{Existence}. For each point $a =(a_1, a_2, \ldots, a_m) \in \mathbb{R}^m$, where $a_i = \pm 1$, consider the function $$1_{a}(x_1, x_2, \ldots, x_m) = \left(\frac{1+a_1 x_1}{2}\right)\left(\frac{1+a_2 x_2}{2}\right) \cdots \left(\frac{1+a_m x_m}{2}\right),$$ where $x_i = \pm 1$. Note that
        \begin{equation*}
            1_{a}(x)=
            \begin{cases}
                $1$, &\text{if } x = a;\\
                $0$, &\text{otherwise.}
            \end{cases}
        \end{equation*}
        Therefore, any function $f:\{-1, +1\}^m \rightarrow \mathbb{R}$ can be written as $$f(x)= {\sum_{a \in \{-1,+1\}^m}} f(a) 1_a(x).$$ Expanding this expression we obtain the required expansion.

        \textit{Uniqueness}. All the functions $f:\{-1, +1\}^m~\rightarrow~\mathbb{R}$ form a \mbox{$2^m$-dimensional} vector space. And there are exactly $2^m$ monomials of a form $\prod_{i \in S} x_i$, where $S~\subset~\{1,2, \ldots, m \}$. Since any function is a linear span of such monomials, it follows that these monomials form a basis. Hence the Fourier-Walsh coefficients are unique.
    \end{proof}

\subsection{A formula for the Fourier-Walsh coefficients}
\label{Fourier subsection 2}
    In this section, we introduce a formula for the Fourier-Walsh coefficients in terms of the expectation of a random variable (Lemma \ref{formula} below).

    \begin{Symbol}
        Let $S \subset \{1, 2, \ldots, m\}$. Denote by $\sigma_{S}$ the function $\sigma_S~\colon~\{-1,+1\}^m~\to~\{-1,+1\}$ given by the formula
        $\sigma_{S}(x_1, \ldots, x_m) = \prod_{i \in S} x_i$.
    \end{Symbol}

    \begin{lemma}
        \label{multiplication of sigmas}
        Let $S,T \subset \{1, 2, \ldots, m\}$. Then $\sigma_{S}\sigma_{T}=\sigma_{S \triangle T}$.
    \end{lemma}

    \begin{proof}
        We have $\sigma_{S}\sigma_{T} = \displaystyle{\prod_{i \in S} x_i \prod_{j \in T} x_i = \prod_{i \in S \cap T} x_i^2 \prod_{j \in S \triangle T} x_i} = \prod_{j \in S \triangle T} x_i = \sigma_{S \triangle T}.$
    \end{proof}

    \begin{lemma}
    \label{independency of expectaion}
        The random variables $\sigma_{\{1\}}, \sigma_{\{2\}}, \ldots, \sigma_{\{m\}}$ have expectation $0$ and are mutually independent.
    \end{lemma}

    The proof of this lemma is obvious.

    \begin{lemma}
        \label{expecation of sigmas}
        Let $S \subset \{1, 2, \ldots, m\}$. Then
        \begin{equation*}
            \mathbb{E}\sigma_{S}=
            \begin{cases}
                $1$, &\text{if } S = \emptyset;\\
                $0$, &\text{otherwise.}
            \end{cases}
        \end{equation*}
    \end{lemma}
    \begin{proof}
        If $S = \emptyset$, then $\mathbb{E}\sigma_{\emptyset} = \mathbb{E}1 = 1$.

        If $S \neq \emptyset$, then by Lemma \ref{independency of expectaion} we have
        $$\mathbb{E}(\sigma_{S}) = \mathbb{E}\prod_{i \in S} x_i = \prod_{i \in S} \mathbb{E}x_i = 0.$$
    \end{proof}

    \begin{lemma}[The formula for the Fourier-Walsh coefficients]
        \label{formula}
        Let $T \subset \{1, 2, \ldots, m\}$. Then for any function $f \colon \{-1,+1\}^m \to \mathbb{R}$ we have $\widehat{f_{T}} = \mathbb{E}(f \sigma_{T})$.
    \end{lemma}

    \begin{proof}
        By Definition \ref{Fourier}, Lemmas \ref{multiplication of sigmas} and \ref{expecation of sigmas} we get
        $$\mathbb{E}(f \sigma_{T}) =
        \mathbb{E}\left[\sigma_{T} \cdot \displaystyle{\sum_{S \subset \{1, 2, \ldots, m\}}} \widehat{f_{S}} \cdot \sigma_{S}\right]
        =\mathbb{E}\left[\displaystyle{\sum_{S \subset \{1, 2, \ldots, m\}}} \widehat{f_{S}} \cdot \sigma_{S \triangle T}\right] = \widehat{f_{T}}.$$
    \end{proof}

\subsection{Variance}
\label{Fourier subsection 3}
    In this section, we prove an important formula, which shows a relation between the variance of a Boolean function and the Fourier-Walsh coefficients of this function (Corollary $\ref{Var}$ below). First we prove a simple fact.

    \begin{lemma}
    \label{Var<=1}
        For any function $f \colon \{-1, +1\}^m \to \{0, 1\}$ we have $0 \leqslant Var(f) \leqslant 1.$
    \end{lemma}

    \begin{proof}
    Obviously, $\mathbb{E}f \in [0,1]$. Hence $Var(f) = \mathbb{E}(f^2)-(\mathbb{E}f)^2 = \mathbb{E}f-(\mathbb{E}f)^2 \in [0,1].$
    \end{proof}

    \begin{lemma}
    \label{Parsevall}
        For any function $f \colon \{-1, +1\}^m \to \mathbb{R}$ we have
        $\mathbb{E}(f^2) = \sum_{S \subset \{1, \ldots, m\}} \widehat{f_S}^2$.
    \end{lemma}

    \begin{proof}
        By Lemmas \ref{multiplication of sigmas} and \ref{expecation of sigmas} we get
        $$\mathbb{E}(f^2) = \mathbb{E}\left[\sum_{S \subset \{1, 2, \ldots, m\}} \widehat{f_{S}} \sigma_S\right]^2 =
        \mathbb{E}\left(\sum_{S \subset \{1, 2, \ldots, m\}} \widehat{f_S}^2 \sigma_S^2 + 2\underset{S \neq T}{\sum_{S, T \subset \{1, 2, \ldots, m\}}} \widehat{f_{S}}\widehat{f_{T}} \sigma_{S \triangle T}\right) = \sum_{S \subset \{1, \ldots, m\}} \widehat{f_S}^2.$$
    \end{proof}

    \begin{corollary}
    \label{Var}
        For any function $f \colon \{-1, +1\}^m \to \mathbb{R}$ we have
        $$Var(f) = \underset{S \neq \emptyset}{\sum_{S \subset \{1, \ldots, m\}}}\widehat{f_S}^2.$$
    \end{corollary}

    \begin{proof}
        Since $\sigma_{\emptyset} = 1$, by Lemmas \ref{formula} and \ref{Parsevall} it follows that
        $$Var(f)=\mathbb{E}(f^2)-(\mathbb{E}f)^2 = \underset{S \neq \emptyset}{\sum_{S \subset \{1, \ldots, m\}}}\widehat{f_S}^2.$$
    \end{proof}

\subsection{Increasing Boolean functions}
    In this section, we consider only functions with the values in $\{0, 1\}$.

    \begin{definition}
    \label{increasing}
        Let $a= (a_1, a_2, \ldots, a_m), b = (b_1, b_2, \ldots, b_m) \in \{-1, 1\}^m$. We write $a \leqslant b$, if for each $i=1, \ldots, m$ we have $a_i \leqslant b_i$. A function $f \colon \{-1, +1\}^m \to \{0, 1\}$ is called \textit{increasing}, if the inequality $a\leqslant b$ implies the inequality $f(a)~\leqslant~f(b).$
    \end{definition}

    \begin{definition}
    \label{i-crit}
        We say that a coordinate $i \in \{1, 2, \ldots, m\}$ is \textit{pivotal} for $f \colon \{-1, +1\}^m \to \{0, 1\}$ on an input $x = (x_1, \ldots, x_m) \in \{-1, +1\}$ if
        $$x_i = +1 \text{ and } f(x_1,\ldots,x_{i-1}, +1, x_{i+1}, \ldots,x_m) \neq f(x_1,\ldots,x_{i-1}, -1, x_{i+1},\ldots,x_m).$$
    \end{definition}

    \begin{remark}
        In the commonly used definition of a pivotal coordinate, one does not require the condition $x_i = +1$. However, it is convenient for us to add this condition. We hope this would not confuse reader.
    \end{remark}




    \begin{lemma}
        \label{formula probability}
        For each increasing function $f \colon \{-1, +1\}^m \to \{0, 1\}$ and each $i \in \{1, \ldots, m\}$ we have
        $$\widehat{f_{\{i\}}} = P(\{x \in \{-1, +1\}^m : \text{the coordinate }i\text{ is pivotal for }f\text{ on the input }x\}).$$
    \end{lemma}

    \begin{proof}
        Denote
        \begin{align*}
        B_+ &= \{x \in \{-1, +1\}^m : f(x) = 1, x_i=+1 \};\\
        B_- &= \{x \in \{-1, +1\}^m : f(x) = 1, x_i=-1 \}.
        \end{align*}

        Consider the injection $g \colon B_- \hookrightarrow{} B_+$ defined by the equality
        $$g(x_1, \ldots, x_{i-1},-1,x_{i+1}, \ldots, x_m) = (x_1, \ldots, x_{i-1},+1,x_{i+1}, \ldots, x_m)$$
        for each $x \in B_-$. Since the function $f$ is increasing, it follows that the map is well-defined. Injectivity is obvious. By Lemma \ref{formula} we have
        \begin{align*}
        \widehat{f_{\{i\}}} &= \mathbb{E}(f(x) x_i) = \frac{|B_+|-|B_-|}{2^m} = \frac{|g(B_-)|+|B_+ \setminus g(B_-)|-|B_-|}{2^m} = \frac{|B_+ \setminus g(B_-)|}{2^m} = \\
        &= \frac{|\{x \in \{-1, +1\}^m : \text{the coordinate }i\text{ is pivotal for }f\text{ on the input }x\}|}{2^m} =\\
        &= P(\{x \in \{-1, +1\}^m : \text{the coordinate }i\text{ is pivotal for }f\text{ on the input }x).
        \end{align*}
    \end{proof}

    \begin{lemma}
    \label{inequality}
        Let $S \subset \{1, 2, \ldots, m\}, S \neq \emptyset$. If a function $f~\colon~\{-1, +1\}^m~\to~\{0, 1\}$ is increasing, then for each $i \in S$ we have $|\widehat{f_S}| \leqslant \widehat{f_{\{i\}}}$.
    \end{lemma}

    \begin{proof}
        Fix $i \in S$. Denote
        \begin{align*}
        C_{++} &= \{x \in \{-1, +1\}^m : f(x) = 1, x_i=+1, \sigma_{S}(x) = +1\};\\
        C_{+-} &= \{x \in \{-1, +1\}^m : f(x) = 1, x_i=+1, \sigma_{S}(x) = -1\};\\
        C_{-+} &= \{x \in \{-1, +1\}^m : f(x) = 1, x_i=-1, \sigma_{S}(x) = +1\};\\
        C_{--} &= \{x \in \{-1, +1\}^m : f(x) = 1, x_i=-1, \sigma_{S}(x) = -1\}.
        \end{align*}

        Note that $\{x \in \{-1, +1\}^m : f(x) = 1\} = C_{++} \sqcup C_{+-} \sqcup C_{-+} \sqcup C_{--}$.

        By Lemma \ref{formula} we have
        \begin{equation}
        \label{f_i=}
        \widehat{f_{\{i\}}} = \mathbb{E}(f x_i ) = \frac{|C_{++} \sqcup C_{+-}| - |C_{--} \sqcup C_{-+}|}{2^m};
        \end{equation}

        \begin{equation}
        \label{f_S=}
        \widehat{f_S} = \mathbb{E}(f \sigma_{S})= \frac{|C_{++} \sqcup C_{-+}| - |C_{+-} \sqcup C_{--}|}{2^m}.
        \end{equation}

        Consider the injection $h \colon C_{-+} \hookrightarrow C_{+-}$ defined by the equality
        $$h(x_1, \ldots, x_{i-1},-1,x_{i+1}, \ldots, x_m) = (x_1, \ldots, x_{i-1},+1,x_{i+1}, \ldots, x_m)$$
        for each $x \in C_{-+}$. Since the function $f$ is increasing and $i\in S$, it follows that the map is well-defined. Injectivity is obvious. Thus
        $|C_{++} \sqcup C_{-+}| \leqslant |C_{++} \sqcup C_{+-}|$ and $|C_{+-}~\sqcup~C_{--}|~\geqslant~|C_{--} \sqcup C_{-+}|$. Comparing (\ref{f_i=}) and (\ref{f_S=}) we get $\widehat{f_S} \leqslant \widehat{f_{\{i\}}}$.

        Analogously one can prove that $\widehat{f_S} \geqslant - \widehat{f_{\{i\}}}$.
    \end{proof}

    By Lemma \ref{inequality} we have the following obvious corollary.
    \begin{corollary}
    \label{maximum}
        Let a function $f~\colon~\{-1, +1\}^m~\to~\{0, 1\}$ be increasing. Then the maximal Fourier-Walsh coefficient, besides $\widehat{f_{\emptyset}}$, is one of $\widehat{f_{\{1\}}}, \ldots, \widehat{f_{\{m\}}}$.
    \end{corollary}

\section{Proof of the theorem}
\subsection{Restatement in terms of the Fourier-Walsh coefficients}
\label{peref}

    To each coloring of the set $M_n$ into 4 colors assign three colorings of this set into 2 colors denoted by $\pm 1$. Table \ref{sop} shows which color is assigned to each cell.

    \begin{minipage}[t]{\textwidth}
    \centering
    \captionof{table}{Construction of new colorings}
    \label{sop}
    \begin{tabular}{|c|c|c|c|}
        \hline
        Coloring into 4 colors & Coloring 1 & Coloring 2 & Coloring 3\\
        \hline
        0 & $+1$ & $+1$ & $+1$\\
        1 & $+1$ & $-1$ & $-1$\\
        2 & $-1$ & $+1$ & $-1$\\
        3 & $-1$ & $-1$ & $+1$\\
        \hline
    \end{tabular}
    \end{minipage}

    \begin{remark}
    \label{sved}
        To each coloring into 4 colors assign the pair (coloring 1, coloring 2). This gives a bijection between the colorings of the set $M_n$:
        $$\{\text{Colorings into 4 colors } 0, 1, 2, 3\} \to \{\text{Colorings into 2 colors } \pm1\}^2.$$
        The bijection preserves the measure on the set of colorings.
    \end{remark}

    \begin{remark}
    \label{3th coloring}
        The color of a cell in the coloring 3 equals the product of colors of the cell in colorings~1 and 2.
    \end{remark}

    \begin{definition}[cf. Definition \ref{4 color sides}]
    \label{2 color sides}
         Fix a coloring of the cells of the set $M_n$ into $2$ colors $\pm1$. We say that in this coloring \textit{there exists percolation between two sides $A$ and $B$ of the hexagon $P$ in $M_n$}, if some cell of the side $A$ is joined with some cell of the side $B$ by a chain of adjacent cells such that each cell in the chain has color $+1$, including the initial cell of the side $A$ and the final cell of the side $B$.
    \end{definition}

    \begin{remark}
    \label{stor}
        Fluid $k$ percolates between two sides $A$ and $B$ of the hexagon in a coloring into 4 colors if and only if in the coloring~$k$ there exists percolation between $A$ and $B$.
    \end{remark}

    Suppose that the set $M_n$ has $m$ cells in total. Label the cells by numbers from $1$ to $m$. In what follows, consider the following notation for colorings of the set $M_n$ into 2 or 4 colors.

    \begin{Symbol}[Colorings into 2 colors]
        To each coloring into 2 colors assign a point $$x = (x_1, \ldots, x_m)~\in~\{-1,+1\}^m.$$
    \end{Symbol}

    \begin{Symbol}[Colorings into 4 colors]
        To colorings into 4 colors assign pairs of colorings into 2 colors by the bijection from Remark \ref{sved}. Denote by $x_i$ and $y_i$ colors in the $i$-th cell in colorings 1 and 2 respectively. To each coloring into 4 colors assign a point
        $$(x,y) = (x_1, \ldots, x_m, y_1, \ldots, y_m)~\in~\{-1,+1\}^m~\times~\{-1,+1\}^m.$$
    \end{Symbol}

    For $k=1, 2, 3$ consider the following functions of the coloring $x \in \{-1,+1\}^m$:
    \begin{equation*}
    f^k(x)=
    \begin{cases}
        \multirow{2}{*}{1,} &\text{if in the coloring $x$ there exists percolation} \\
                            &\text{between sides $k$ and $k+3$ of the hexagon $P$;}\\
        0, &\text{otherwise.}
    \end{cases}
    \end{equation*}

    \begin{Symbol}
        Let $S \subset \{1, 2, \ldots, m\}$. Denote by $\sigma_{S}$ and $\tau_S$ functions \\ $\{-1,+1\}^m~\times~\{-1,+1\}^m~\to~\{-1,+1\}$ defined by equations
        \begin{align*}
        \sigma_{S}(x, y) &= \sigma_{S}(x_1, \ldots, x_m, y_1, \ldots, y_m) = \prod_{i \in S} x_i,\\
        \tau_{S}(x, y) &= \tau_{S}(x_1, \ldots, x_m, y_1, \ldots, y_m) = \prod_{i \in S} y_i.
        \end{align*}
    \end{Symbol}

    We state a few properties of functions $\sigma_S$ and $\tau_S$. The proofs are analogous to the proofs of Lemmas \ref{multiplication of sigmas}-\ref{expecation of sigmas}.

    \begin{lemma}
    \label{sigma tau}
        Let $S,T \subset \{1, 2, \ldots, m\}$. Then:\\
        a) $\sigma_{S}\sigma_{T}=\sigma_{S \triangle T}$ and $\tau_{S}\tau_{T}=\tau_{S \triangle T}$.\\
        b) The random variables $\sigma_{\{1\}}, \sigma_{\{2\}}, \ldots, \sigma_{\{m\}}, \tau_{\{1\}}, \tau_{\{2\}}, \ldots, \tau_{\{m\}}$ (on the space $\{-1,+1\}^m\times\{-1,+1\}^m)$ have expectation 0 and are mutually independent.\\
        c)
            $\mathbb{E}(\sigma_{S}\tau_{T})=
            \begin{cases}
                $1$, &\text{if } S = T = \emptyset;\\
                $0$, &\text{otherwise.}
            \end{cases}$
    \end{lemma}

    Now we prove the main lemma of this subsection.

    \begin{lemma}
    \label{formulae for sides probability}
     $P(B_{k,n}) = \widehat{f^k_{\emptyset}}$ for each $k = 1, 2, 3$, and $P(B_{1,n} \cap B_{2,n} \cap B_{3,n}) =~\displaystyle{\sum_{S \subset \{1, \ldots, m\}}~\widehat{f^1_S} \widehat{f^2_S} \widehat{f^3_S}}.$
     \end{lemma}

    \begin{proof}
        In the following computations, expectation in the proof is taken in the space $\{-1,+1\}^m$ in the first three formulae and in the space $\{-1,+1\}^m\times\{-1,+1\}^m$ in the last formula. We denote $f^3(x y) = f^3(x_1 y_1, \ldots, x_m y_m)$.

        By Definition \ref{Fourier}, Lemmas \ref{formula} and \ref{sigma tau}, Remarks \ref{stor}, \ref{sved}, and \ref{3th coloring} we have
        \begin{align*}
            P(B_{1,n}) &= \frac{|\{(x,y) \in \{-1,+1\}^m\times\{-1,+1\}^m : f^1(x) = 1\}|}{4^m} =
            \frac{2^m |\{x \in \{-1,+1\}^m : f^1(x) = 1\}|}{4^m} = \\
            &= \mathbb{E} f^1 = \widehat{f^1_{\emptyset}};\\
            P(B_{2,n}) &= \frac{|\{(x,y) \in \{-1,+1\}^m\times\{-1,+1\}^m : f^2(y) = 1\}|}{4^m} =
            \frac{2^m |\{y \in \{-1,+1\}^m : f^1(y) = 1\}|}{4^m} = \\
            &= \mathbb{E} f^2 = \widehat{f^2_{\emptyset}};\\
            P(B_{3,n}) &= \frac{|\{(x,y) \in \{-1,+1\}^m\times\{-1,+1\}^m : f^3(x y) = 1\}|}{4^m} =
            \frac{2^m |\{z \in \{-1,+1\}^m : f^3(z) = 1\}|}{4^m} = \\
            &= \mathbb{E} f^3 = \widehat{f^3_{\emptyset}}.
        \end{align*}
        (Note that the proof of the third formula is slightly different from the first two proofs.)

        \begin{align*}
        &P(B_{1,n} \cap B_{2,n} \cap B_{3,n}) = \frac{|\{(x,y)\in\{-1,+1\}^m\times\{-1,+1\}^m : f^1(x) f^2(y) f^3(x y) = 1\}|}{4^m} = \\
        &= \mathbb{E} \left[f^1(x) f^2(y) f^3(x y)\right] = \mathbb{E}\left[\sum_{S \subset \{1, \ldots, m\}} \widehat{f^1_{S}} \sigma_S \sum_{S \subset \{1, \ldots, m\}} \widehat{f^2_{S}} \tau_S \sum_{S \subset \{1, \ldots, m\}} \widehat{f^3_{S}} \sigma_S \tau_S\right] =\\
        &= \mathbb{E}\left[\sum_{S \subset \{1, \ldots, m\}}~\widehat{f^1_S} \widehat{f^2_S} \widehat{f^3_S} \sigma_S^2 \tau_S^2 + \underset{\text{not all }S, T, R \text{ are equal}}{\sum_{S, T, R \subset \{1, \ldots, m\}}} \widehat{f^1_S} \widehat{f^2_T} \widehat{f^3_R} \sigma_{S \triangle R}  \tau_{T \triangle R}\right] =
        \sum_{S \subset \{1, \ldots, m\}}~\widehat{f^1_S} \widehat{f^2_S} \widehat{f^3_S}.
        \end{align*}
    \end{proof}

    \begin{corollary}
    \label{sides refolmulation}
       $P(B_{1,n} \cap B_{2,n} \cap B_{3,n}) -  P(B_{1,n}) P(B_{2,n}) P(B_{3,n}) = \displaystyle{\underset{S \neq \emptyset}{\sum_{S \subset \{1, \ldots, m\}}}}~\widehat{f^1_S} \widehat{f^2_S} \widehat{f^3_S}.$
    \end{corollary}

    Thus for the proof of Theorem \ref{sides} it remains to prove that $$\lim_{n \to \infty}\underset{S \neq \emptyset}{\sum_{S \subset \{1, \ldots, m\}}}~\widehat{f^1_S} \widehat{f^2_S} \widehat{f^3_S} = 0.$$

\subsection{Pivotal colorings}

    Obviously, the functions $f^1, f^2, f^3$ are increasing. Therefore, all the lemmas from \S \ref{section fourier} hold for these functions. In the following lemma we make use of this.

    \begin{lemma}
    \label{ineq}
        $$\left|\underset{S \neq \emptyset}{\sum_{S \subset \{1, \ldots, m\}}}~\widehat{f^1_S} \widehat{f^2_S} \widehat{f^3_S}\right| \leqslant \max_{i \in \{1, \ldots, m \}} P(\{x \in \{-1,+1\}^m :\text{the coordinate }i\text{ is pivotal for }f^1\text{ on the input }x\}).$$
    \end{lemma}

    \begin{proof} By the Cauchy-Bunyakovsky-Schwarz inequality, Corollary \ref{Var}, and Lemmas \ref{Var<=1}, \ref{inequality}, and \ref{formula probability} we have
        \begin{align*}
        \left|\underset{S \neq \emptyset}{\sum_{S \subset \{1, \ldots, m\}}}~\widehat{f^1_S} \widehat{f^2_S} \widehat{f^3_S}\right| &\leqslant
        \underset{S \neq \emptyset}{\sum_{S \subset \{1, \ldots, m\}}}~|\widehat{f^2_S}| |\widehat{f^3_S}| |\widehat{f^1_S}| \leqslant
        \max_{S \neq \emptyset}{|\widehat{f^1_S}|} \underset{S \neq \emptyset}{\sum_{S \subset \{1, \ldots, m\}}}~|\widehat{f^2_S}| |\widehat{f^3_S}| \leqslant\\
        &\leqslant \max_{S \neq \emptyset}{|\widehat{f^1_S}|} \sqrt{Var(f^2) Var(f^3)} \leqslant
        \max_{S \neq \emptyset}{|\widehat{f^1_S}|} =
        \max_{i \in \{1, \ldots, m \}} {\widehat{f^1_{\{i\}}}} = \\
        &= \max_{i \in \{1, \ldots, m \}} P(\{x \in \{-1,+1\}^m :\text{the coordinate }i\text{ is pivotal for }f^1\text{ on the input }x\}).
        \end{align*}
    \end{proof}

    \begin{lemma} For each cell $i \in \{1, \ldots, m\}$ of the set $M_n$ we have
    \label{ineq2}
        \begin{align*}
            P(\{x \in &\{-1,+1\}^m :\text{the coordinate }i\text{ is pivotal for }f^1\text{ on the input }x\}) \leqslant \\
             \leqslant \min\{&P(\{x \in \{-1,+1\}^m :\text{in the coloring } x  \text{ there exists percolation from the cell }i\text{ to the side }1\}),\\
             &P(\{x \in \{-1,+1\}^m :\text{in the coloring } x  \text{ there exists percolation from the cell }i\text{ to the side }4\})\}.
        \end{align*}
    \end{lemma}

    \begin{proof}
        If the coordinate $i$ is pivotal for $f^1$ on the input $x$, then in the coloring $x$, a cell of the side~$1$ is joined with some cell of the side~$4$ by a chain of adjacent cells such that each cell in the chain has color $+1$. If the color of the cell $i$ is changed, then percolation disappears. Hence the chain must contain the cell $i$, and thus there is a chain of color $+1$ from the cell $i$ to the side~$1$, and a chain of color $+1$ from the cell $i$ to the side~$4$. This implies the required inequality.
    \end{proof}

\subsection{Conclusion of the proof}

    Let $i \in M_n$ be a cell (see Figure 2a). Denote by $r(i)$ the maximal distance from the center of the cell $i$ to the lines containing the sides $1$ and $4$ of the hexagon $P$. Denote by $P(i)$ the rectangle $2r(i) \times 8r(i)$ centered at the center of the cell~$i$, with the larger side containing one of the side $1$ or~$4$ of the hexagon~$P$ (See Figure 2b). Without loss of generality we assume that the larger side of the rectangle contains the side $1$ of the hexagon.

    \begin{lemma}
    \label{rectangle}
        The hexagon $P$ is entirely contained inside the rectangle $P(i)$. 
    \end{lemma}

    \begin{proof}
        (See Figure 2c). We introduce the Cartesian coordinate system centered at the center of the cell $i$, such that the $x$-axis is parallel to the line containing the side $1$. Then the rectangle $2r(i) \times 8r(i)$ is given by the system
        \begin{equation}
        \label{rect}
            \begin{cases}
                -4r(i) \leqslant x \leqslant 4r(i),\\
                -r(i) \leqslant y \leqslant r(i).
            \end{cases}
        \end{equation}
        The "long diagonal" (see Figure~2d) of the hexagon~$P$ equals $2$, hence the hexagon~$P$ is contained inside the rectangle
        \begin{equation}
        \label{rect2}
            \begin{cases}
                -2 \leqslant x \leqslant 2,\\
                -r(i) \leqslant y \leqslant r(i).
            \end{cases}
        \end{equation}
        The "small diagonal" (see Figure~2d) of the hexagon~$P$ equals $\sqrt{3}$, hence $r(i)\geqslant\sqrt{3}/2$. Therefore, if a point $(x, y)$ satisfies system~(\ref{rect2}), then it also satisfies system~(\ref{rect}).
    \end{proof}

    \begin{figure}[h]
     \centering
     \begin{subfigure}[t]{0.26\textwidth}
        \raisebox{-\height}{\includegraphics[width=\textwidth]{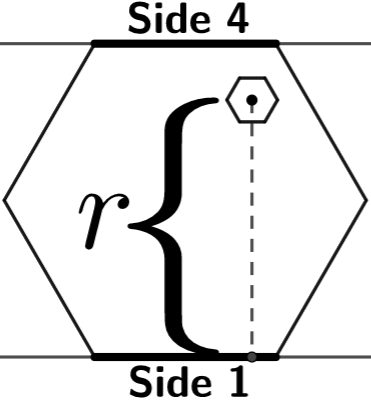}}
        \caption{}
     \end{subfigure}
    \hfill
     \begin{subfigure}[t]{0.32\textwidth}
        \raisebox{-\height}{\includegraphics[width=\textwidth]{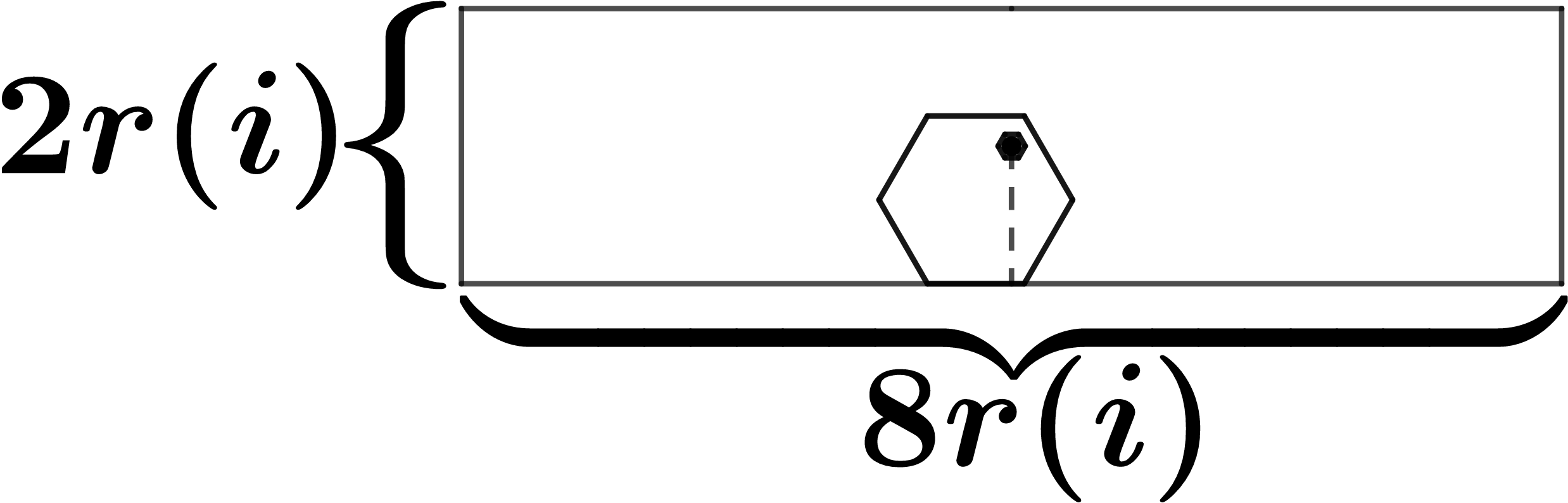}}
        \caption{}
        \vspace{.3ex}
        \raisebox{-\height}{\includegraphics[width=\textwidth]{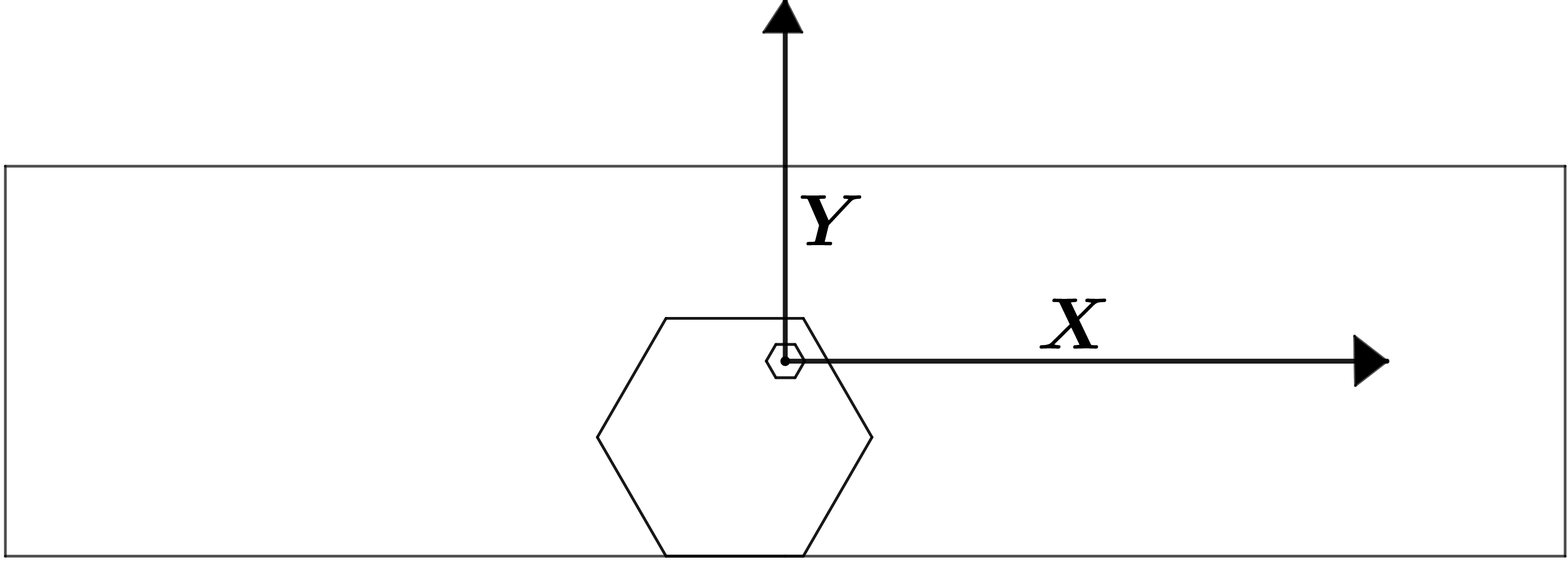}}
        \caption{}
     \end{subfigure}
    \hfill
     \begin{subfigure}[t]{0.36\textwidth}
        \raisebox{-\height}{\includegraphics[width=\textwidth]{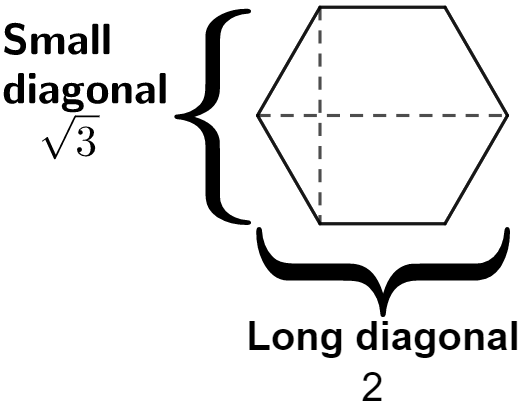}}
     \caption{}
     \end{subfigure}
     \caption{}
    \end{figure}

    To the set $M_n$ add all the cells that are entirely contained in the rectangle $P(i)$. Denote by $M'_n$ the resulting set. In what follows, the colorings of the set $M'_n$ into 2 colors $\pm 1$ are denoted by points in the set $\{-1,+1\}^{|M'_n|}$.

    \begin{lemma}
    \label{ineq3}
        The following inequalities hold:
        \begin{align*}
            &P(\{x \in \{-1,+1\}^m : \text{in the coloring } x  \text{ there exists percolation from the cell }i\text{ to the side }1\}) \leqslant\\
            &P(\{x \in \{-1,+1\}^{|M'_n|} : \text{in the coloring } x  \text{ there exists percolation from the cell }i\text{ to the side }1\}) \leqslant\\
            &P(\{x \in \{-1,+1\}^{|M'_n|} : \text{in the coloring } x  \text{ there exists percolation from the cell }i\text{ to the boundary of } M'_n\}).
        \end{align*}
    \end{lemma}

    \begin{proof}
        By Lemma \ref{rectangle} we have $M_n \subset M'_n$. If there exists percolation in the coloring $x \in \{-1,+1\}^m$, then there exists percolation in all the colorings $y \in \{-1,+1\}^{|M'_n|}$ coinciding with the coloring~$x$ on the set~$M_n$. This implies the first inequality.

        If in some coloring of the set $M_n$ there exists percolation from the cell $i$ to the side $1$, then by Lemma \ref{rectangle} in this coloring there exists percolation from the cell $i$ to a cell of the boundary of the set~$M'_n$. This implies the second inequality.
    \end{proof}

    \begin{lemma}
    \label{osnova}
        We have
        \begin{align*}
            \min\{&P(\{x\in\{-1,+1\}^m : \text{in the coloring } x  \text{ there exists percolation from the cell }i\text{ to the side }1\}), \\
            &P(\{x\in\{-1,+1\}^m : \text{in the coloring } x  \text{ there exists percolation from the cell }i\text{ to the side }4\})\}           \leqslant \\
            P(\{&\text{in the set } M_n \text{ there exists percolation from the cell } O \text{ to the boundary}\}).
        \end{align*}
    \end{lemma}

    \begin{proof}
        Consider the regular hexagon centered at the center of the cell~$i$, obtained from the hexagon~$P$ by translation. Since $r(i)~\geqslant~\sqrt{3}/2$ (see. Figure~2d), it follows that the resulting hexagon is entirely contained inside the rectangle~$P(i)$. By Lemma~\ref{ineq3} this implies the required inequality.
    \end{proof}

    \begin{proof}[Proof of Theorem 1]
    Theorem 1 follows directly from Corollary \ref{sides refolmulation}, Lemmas \ref{ineq}, \ref{ineq2}, and \ref{osnova}, and Theorem \ref{kesten lite}.
    \end{proof}

\subsection{Arguments in favor of Conjecture \ref{hyp2}}
    In conclusion, we give informal arguments in favor of Conjecture \ref{hyp2}. In the (degenerate) case when the hexagon $P$ consists of a single cell, we have $\frac{P(A_{1,n} \cap A_{2,n} \cap A_{3,n})}{P(A_{1,n}) P(A_{2,n}) P(A_{3,n})}=1$. In Example~1, the set~$M_n$ is larger, and we have $\frac{P(A_{1,n} \cap A_{2,n} \cap A_{3,n})}{P(A_{1,n}) P(A_{2,n}) P(A_{3,n})}>1$. A natural conjecture is that as the number of the cells in the set $M_n$ increases, the ratio in question also increases. We also give the following more solid argument by K.~Izyurov.

    Consider the following function of a coloring $x \in \{-1,+1\}^m$:
    \begin{equation*}
    \chi(x) = \chi(x_1, \ldots, x_m)=
    \begin{cases}
        \multirow{2}{*}{1,} &\text{if in the coloring $x$ there exists percolation between } \\
                            &\text{the center and the boundary of the set $M_n$;}\\
        0, &\text{otherwise.}
    \end{cases}
    \end{equation*}

    Now we state a lemma and a corollary. (The proofs are analogous to the proof of Lemma \ref{formulae for sides probability}.)
    \begin{lemma}
    \label{center fourier}
     $P(A_{1,n}) = P(A_{2,n}) = P(A_{3,n}) = \widehat{\chi_{\emptyset}}$ and $P(A_{1,n} \cap A_{2,n} \cap A_{3,n}) =~\displaystyle{\sum_{S \subset \{1, \ldots, m\}}~\widehat{\chi_S}^3}.$
    \end{lemma}

    \begin{corollary}
    $P(A_{1,n} \cap A_{2,n} \cap A_{3,n}) -  P(A_{1,n}) P(A_{2,n}) P(A_{3,n}) =
    \displaystyle{\underset{S \neq \emptyset}{\sum_{S \subset \{1, \ldots, m\}}} \widehat{\chi_S}^3}.$
    \end{corollary}
    Thus Conjectures \ref{hyp1} and \ref{hyp2} are equivalent to the inequalities
    $$\displaystyle{\underset{S \neq \emptyset}{\sum_{S \subset \{1, \ldots, m\}}} \widehat{\chi_S}^3} > 0\text{ and }\displaystyle{\lim_{n \to +\infty}}{\underset{S \neq \emptyset}{\sum_{S \subset \{1, \ldots, m\}}} \left(\frac{\widehat{\chi_S}}{\widehat{\chi_{\emptyset}}}\right)}^3 > 0.$$

    The latter sum contains
    $$\sum_{\text{$s$ is adjacent to the center $O$}} \left(\frac{\widehat{\chi_S}}{\widehat{\chi_{\emptyset}}}\right)^3.$$
    Consider some cell $i$ among the 6 ones adjacent to the center $O$. The coefficient $\chi_{\{i\}}$ equals the probability that the coordinate $i$ is pivotal. It is easy to prove that the probability that the coordinate $i$ is pivotal conditioned by the existence of percolation (i.e., $\chi_{\{i\}} / \chi_{\emptyset}$) is small, but positive and does not tend to $0$ as $n \to \infty$. Hence, these 6 summands make a contribution that does not disappear in the limit (and the desired limit can be equal to $0$, only if this contribution is canceled with other Fourier-Welsh coefficients, what no reason can be seen for). Since a contribution is small, it is not noticed in numerical experiments.

\section{Numerical experiments}
\label{computer}
    In order to verify Conjectures \ref{hyp1}-\ref{hyp2} about percolation from the center, a computer program was written. (See C++ code in \cite{programm}.)

    Each side of the hexagon $P$ is taken to be parallel to some side of a lattice cell. The program considers $k$ random colorings of the polygon $M_n$. Depending on input $p = 1, 2,$ or $3$, it computes the number $T_p$ of colorings with percolation of all the fluids $1, \ldots, p$. (For example, if $p=2$, then the program computes the number of colorings with percolation of both fluids $1$ and $2$.)\\

    \begin{center}
    \begin{tabular}{|c|c|}
        \hline
        Input & Output\\
        \hline
        $n$ & $T_p$\\
        $k$ & \\
        $p$ & \\
        \hline
    \end{tabular}
    \end{center}

    Output results for some $n$, $k$, and $p$ are shown in Table \ref{program}. Additionally, the value $P = (T_3/k)/(T_1/k)^3 \approx \frac{P(A_{1,n} \cap A_{2,n} \cap A_{3,n})}{P(A_{1,n}) P(A_{2,n}) P(A_{3,n})}$ is computed. Presence of values $P < 1$ does not disprove Conjecture \ref{hyp1}, because it can be caused by statistical deviations. For $n=500, k=1 000 000, p=1$ the program runs approximately 2 hours.

    \begin{table}[H]
    \caption{Output results}
    \begin{center}
    \label{program}
    \begin{tabular}{|c|c|c|c|c|c|c|}
        \hline
        n & k & $T_1$ & $T_2$ & $T_3$ & $P$ & $P-1$\\
        \hline
        3& 10 000 000& 9 844 112& 9 690 354& 9 538 785& 1.000 & <0.001\\
        4& 10 000 000& 9 843 390& 9 691 475& 9 537 528& 1.000 & <0.001\\
        5& 10 000 000& 9 655 567& 9 327 500& 9 009 425 & 1.001 & 0.001\\
        6& 10 000 000& 9 575 758& 9 166 785 & 8 779 321& 1.000 & <0.001\\
        7& 10 000 000& 9 478 577& 8 985 300& 8 522 311& 1.001 & 0.001\\
        8& 10 000 000& 9 368 765& 8 778 158& 8 229 285& 1.001 & 0.001\\
        9& 10 000 000& 9 285 683& 8 630 914& 8 016 759& 1.001 & 0.001\\
        10& 10 000 000& 9 211 830& 8 475 622& 7 813 657& 1.000 & <0.001\\
        15& 10 000 000& 8 898 951& 7 893 907& 7 030 482& 0.998 & -0.002\\
        20& 10 000 000& 8 657 067& 7 466 236& 6 462 095& 0.996& -0.004\\
        25& 10 000 000& 8 461 617& 7 166 402& 6 016 836& 0.993& -0.007\\
        50& 1 000 000& 787 626& 621 883& 490 967& 1.005& 0.005\\
        100& 1 000 000& 733 994& 539 606& 395 491& 1.000 & <0.001\\
        150& 1 000 000& 707 948& 500 289& 354 320& 0.999 & -0.001\\
        200& 1 000 000& 681 800& 466 133& 318 058& 1.004 & 0.004\\
        300& 1 000 000& 660 490& 429 462& 283 512& 0.984 & -0.016\\
        400& 1 000 000& 633 621& 402 712& 259 337& 1.019 & 0.019\\
        500& 1 000 000& 621 187& 386 527& 242 821& 1.013& 0.013\\
        1500& 1 000 000& 555 244& 299 740& 166 587& 0.973 & -0.027\\
        \hline
    \end{tabular}
    \end{center}
    \end{table}

    Let us describe the algorithm checking whether fluid $1$ percolates from the center $O$ to the boundary in a given coloring. Since the color of the center is irrelevant, we may assume that the center has color $0$. (See Table \ref{examples} for examples.)

    Fix a coloring. At the beginning, a beetle sits at the center. The beetle is allowed to move to an adjacent cell with color $0$ or $1$. We are going to find out if the beetle can reach the boundary.

    \setlength{\parindent}{0pt}

    \textbf{Step~1.} If the beetle is already in a boundary cell, then go to END1. Otherwise go to Step~2.

    \textbf{Step~2.} If the color of the cell below the beetle is $0$ or $1$, then move the beetle downwards, and go to Step~1.
    If the color of the cell below the beetle is $2$ or $3$, then go to Step~3.

    \textbf{Step~3.} (Building of a wall; see Table \ref{examples})

    Construct a finite sequence (called a \textit{wall} in what follows) of distinct adjacent sides of cells (called \textit{wall segments}) as follows. The first term of the sequence is the bottom side of the cell where the beetle is located, and the second term of the sequence is adjacent to the left endpoint of the first wall segment. All the cells from one side of the wall have color $0$ or $1$, and all the cells from the other side of the wall have color $2$ or $3$. And finally, either the last wall segment lies between two boundary cells or all the wall segments surround some collection of cells of $M_n$. In the former case go to END1 and in the latter case go to Step~4.

    (It is easy to see that the wall is uniquely defined by the conditions above.)

    \textbf{Step~4.} Draw the ray from the center of $M_n$ to the bottom. If the ray intersects with wall segments (which were constructed in Step~3) an odd number of times, then go to END2. Otherwise, move the beetle  to the cell right below the lowest wall segment intersecting the ray, and go to Step~1.

    \textbf{END1.} There is a percolation from the center to the boundary.

    \textbf{END2.} There is no percolation from the center to the boundary.

    The following table shows how the algorithm works. In all three examples $n=7$.

    \begin{table}[H]
    \begin{center}
    \caption{Percolation-checking algorithm}
    \label{examples}
    \begin{tabular}{| p{4.1cm} | p{4.1cm} | p{4.1cm} | p{4.1cm} |}
        \hline
        \multicolumn{4}{|c|}{\includegraphics[trim=0 60 0 -1, width=0.75\textwidth]{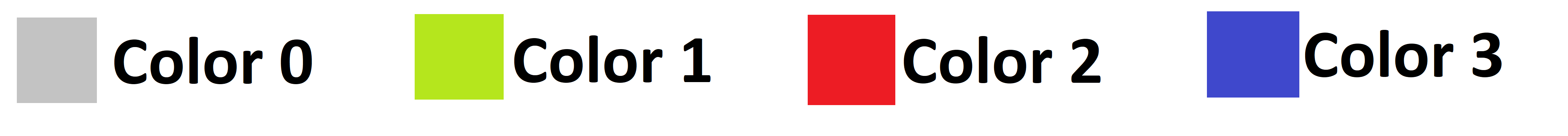}}\\
        \hline
        \multicolumn{4}{|c|}{Example 1} \\
        \hline
            \includegraphics[trim=0 0 0 -1,width=0.24\textwidth]{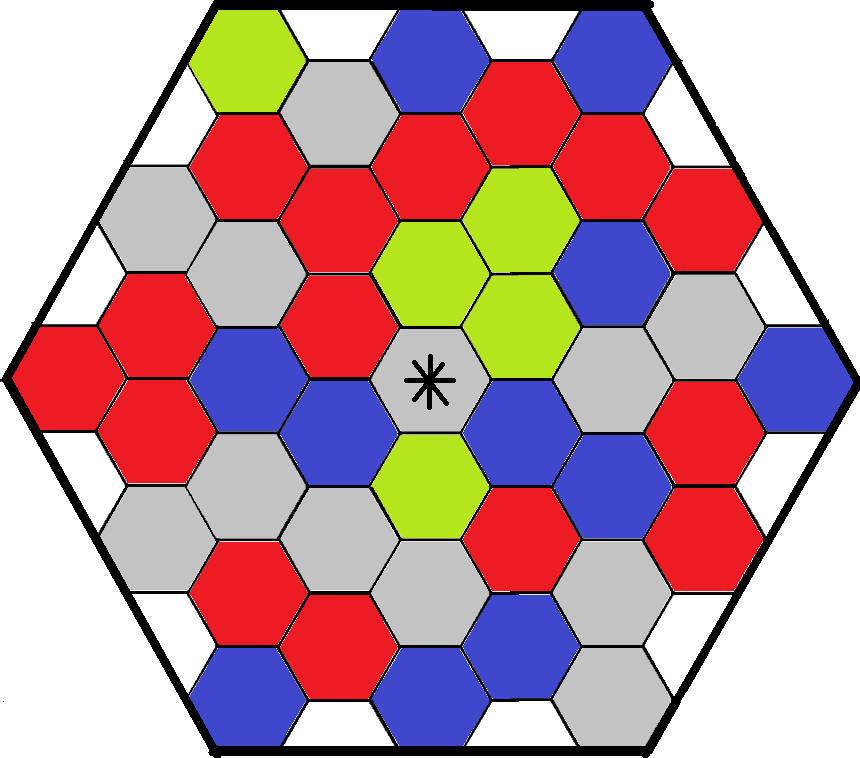}
            & \includegraphics[width=0.24\textwidth]{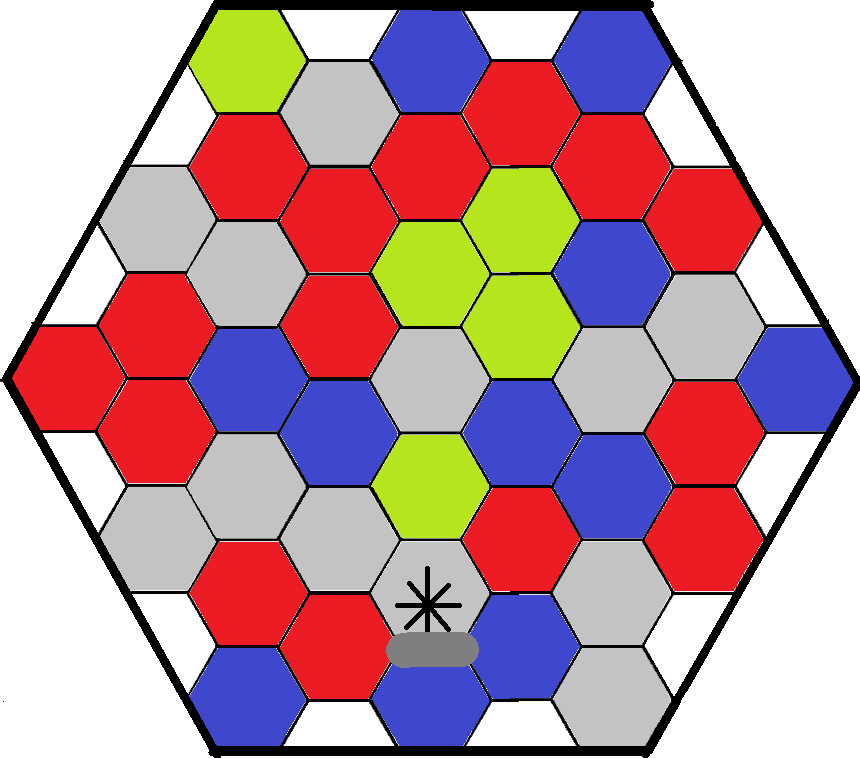}
            & \includegraphics[width=0.24\textwidth]{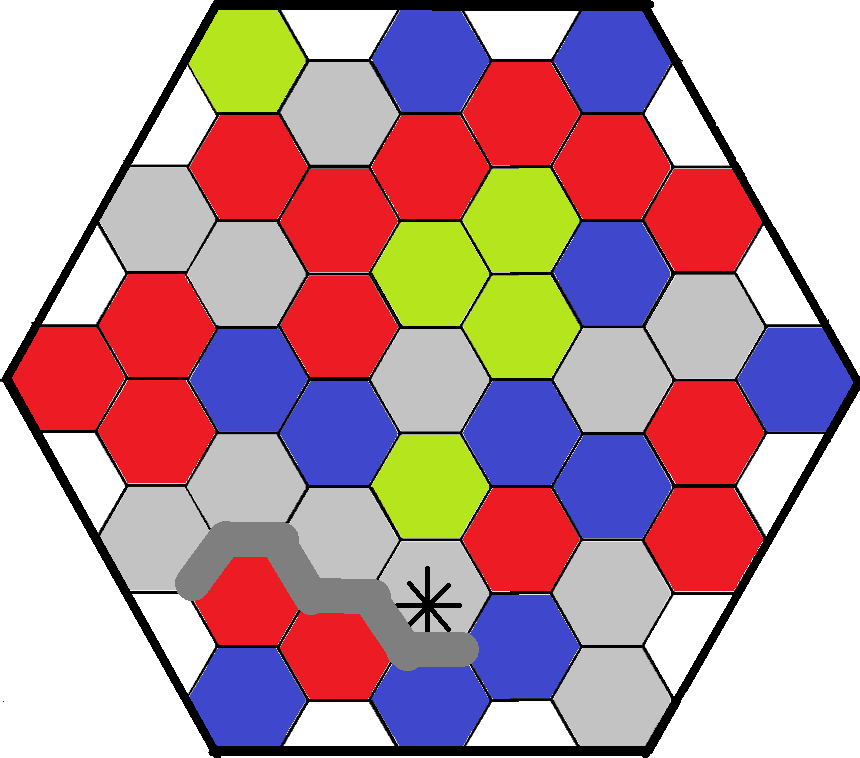}
            & \includegraphics[width=0.24\textwidth]{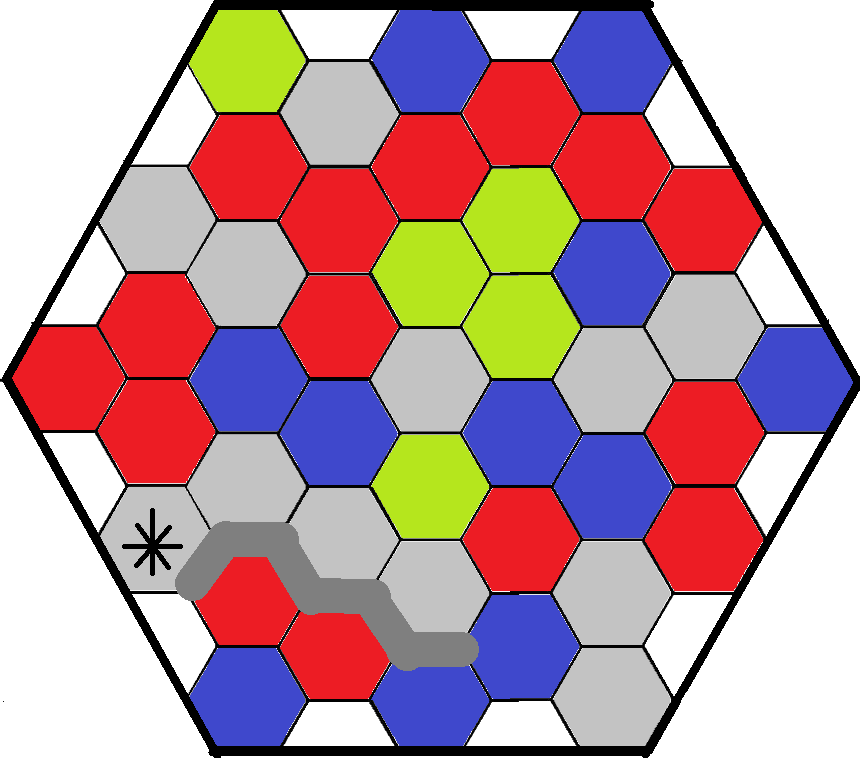}\\
        \hline
        At the beginning, the beetle is sits at the center.&
        The beetle moves downwards until it hits a blue cell. We put a wall segment right below the beetle.&
        We build the wall until we put a wall segment between two boundary cells.&
        The beetle can move along the wall to reach the boundary.\\
        \hline
        \multicolumn{4}{|c|}{Example 2} \\
        \hline
            \includegraphics[trim=0 0 0 -2, width=0.24\textwidth]{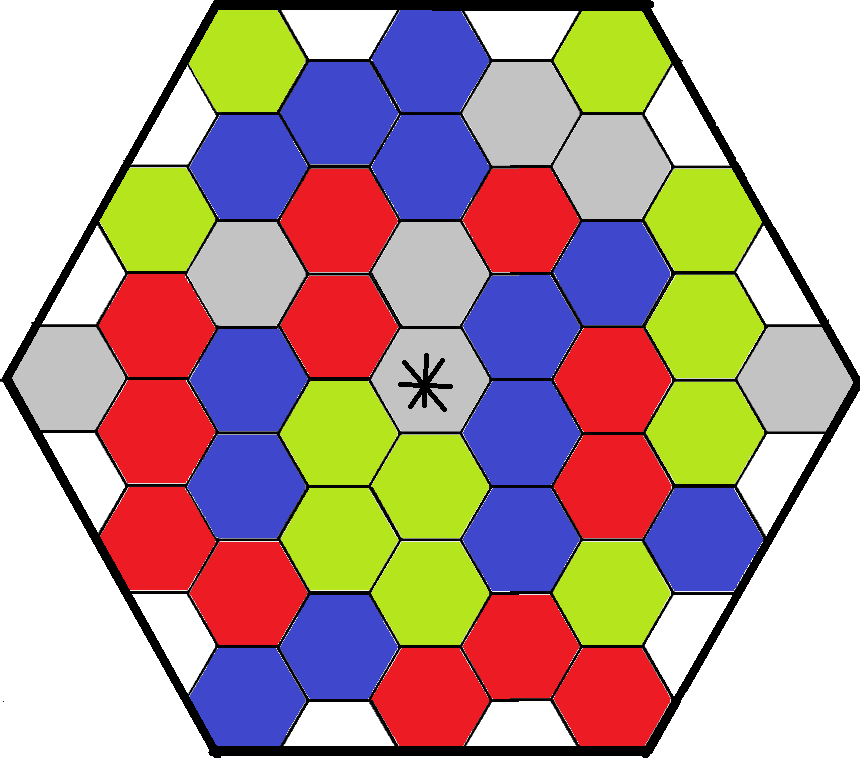}
            & \includegraphics[width=0.24\textwidth]{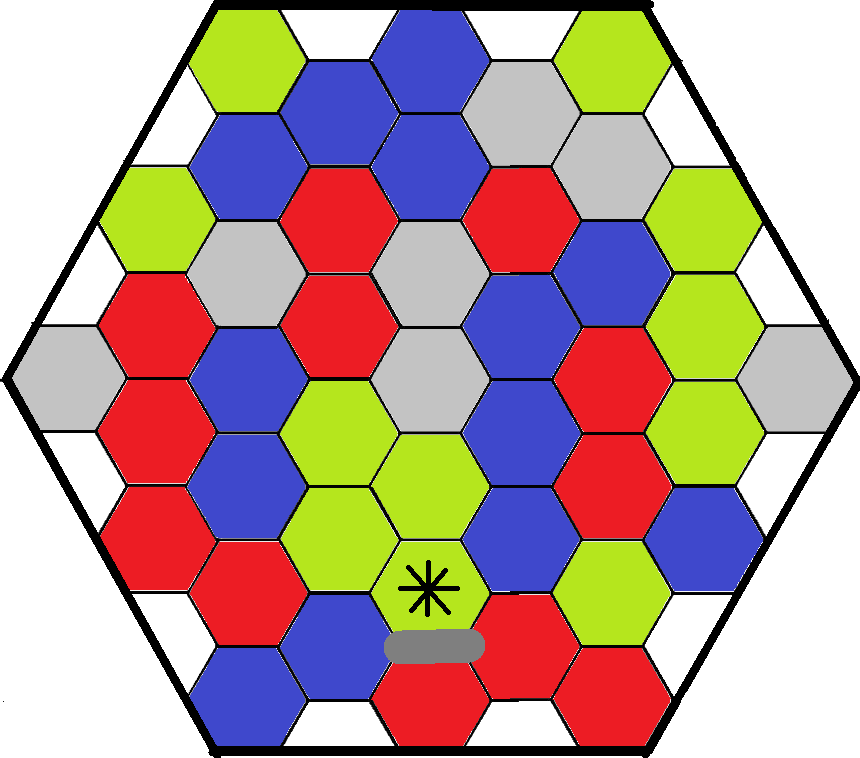}
            & \includegraphics[width=0.24\textwidth]{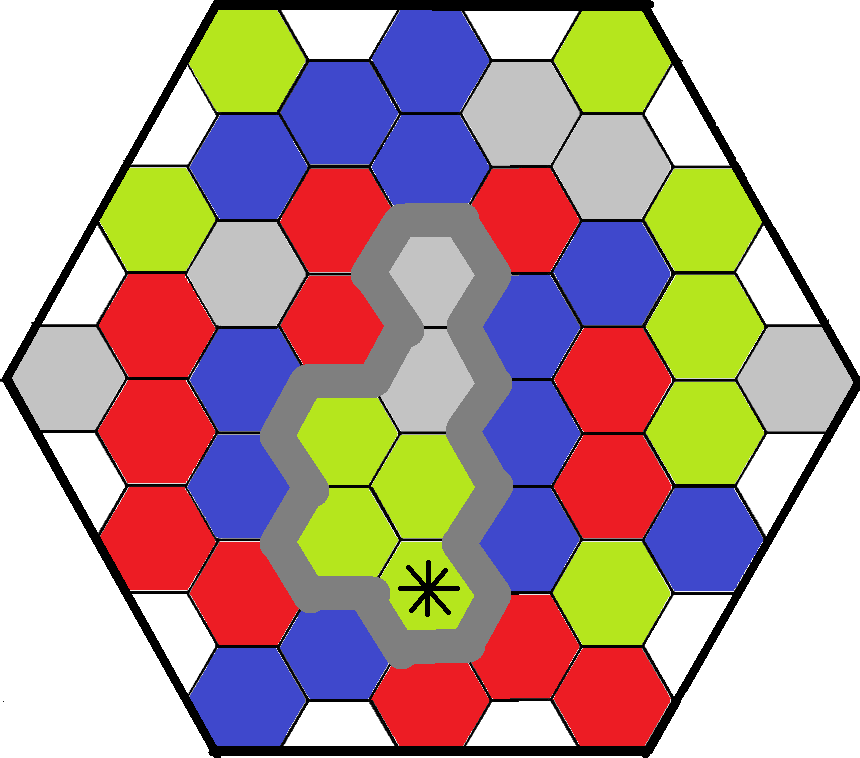}
            & \includegraphics[width=0.24\textwidth]{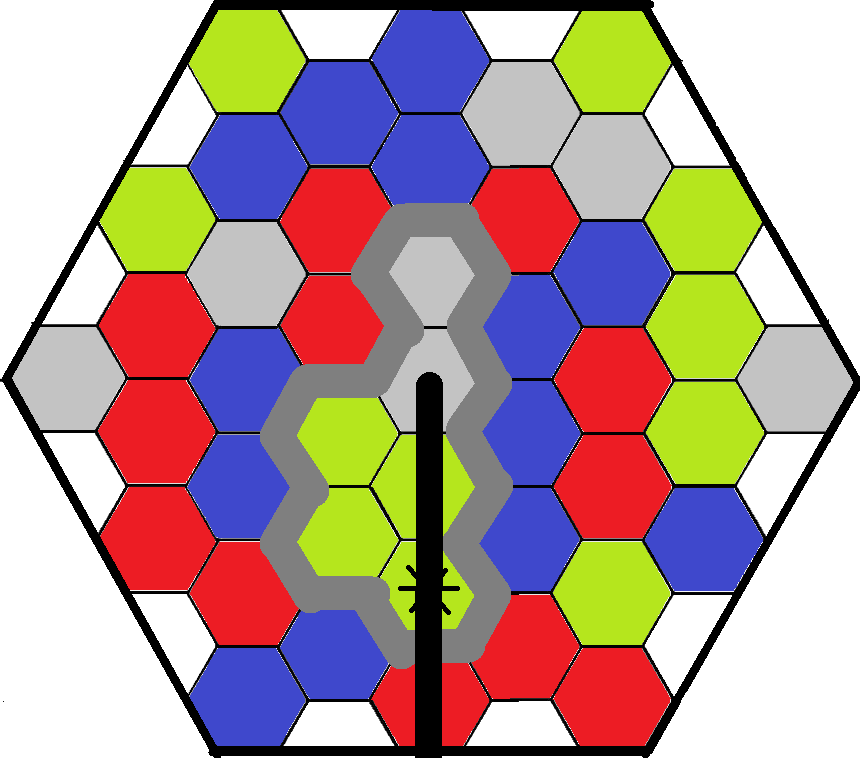}\\
        \hline
        At the beginning, the beetle sits at the center.&
        The beetle moves downwards until it hits a red cell. We put a wall segment right below the beetle.&
        We build the wall until it encloses some area of $M_7$.&
        We draw the ray from the center to the bottom. The ray intersects a unique wall segment. Hence the beetle cannot reach the boundary.\\
        \hline
        \multicolumn{4}{|c|}{Example 3} \\
        \hline
            \includegraphics[trim=0 0 0 -2,width=0.24\textwidth]{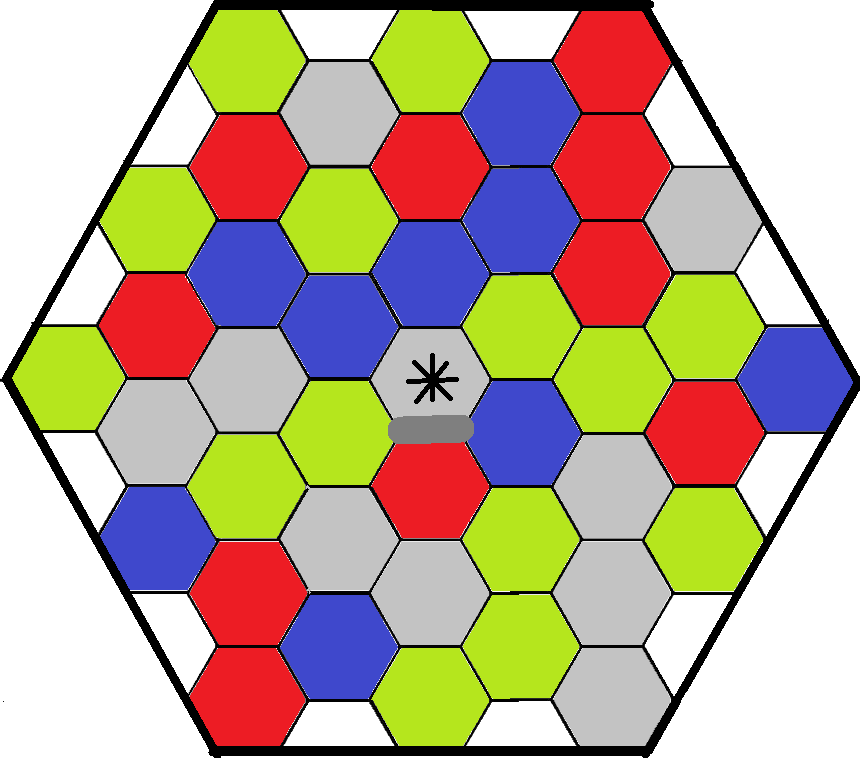}
            & \includegraphics[width=0.24\textwidth]{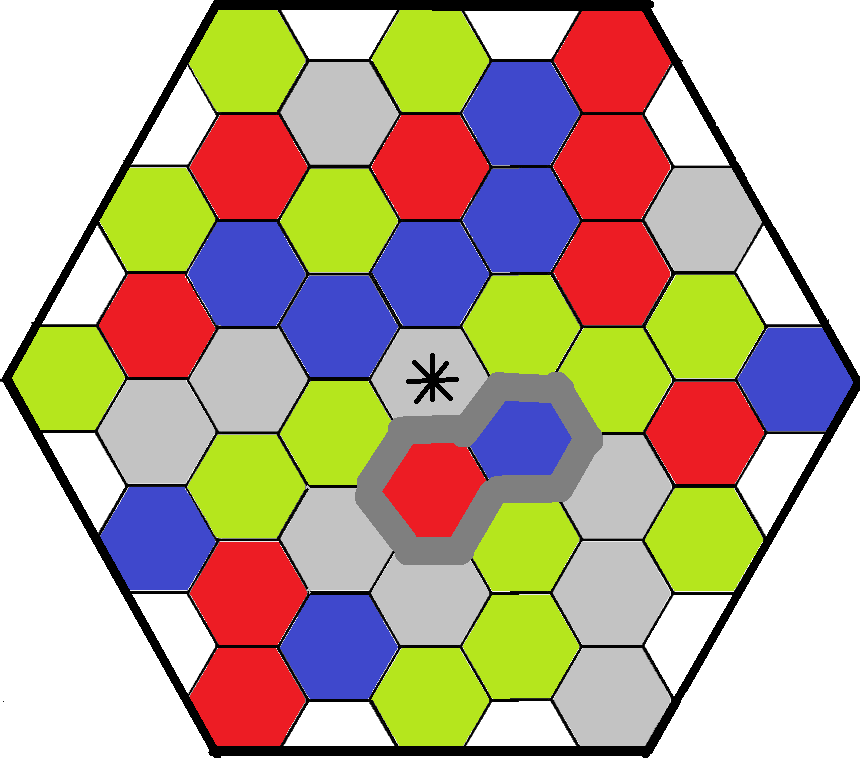}
            & \includegraphics[width=0.24\textwidth]{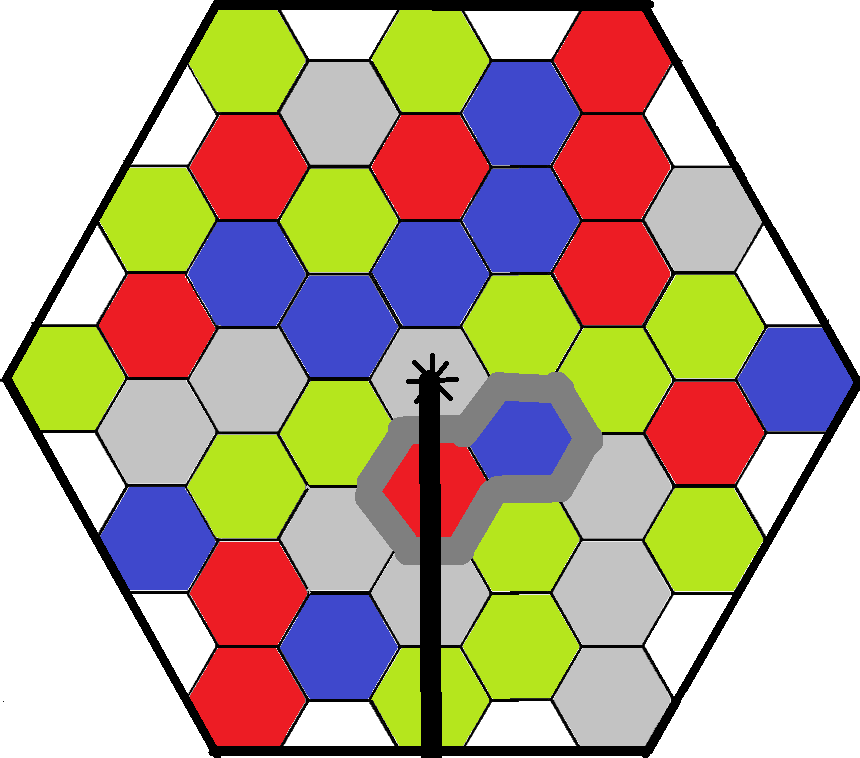}
            & \includegraphics[width=0.24\textwidth]{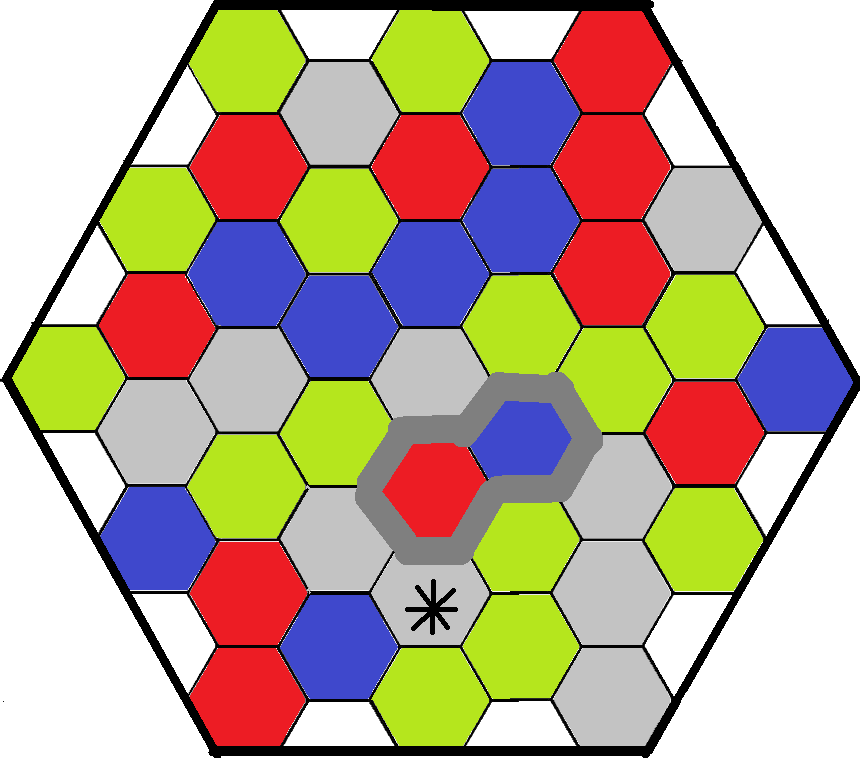}\\
        \hline
        At the beginning, the beetle sits at the center. It cannot move downwards, thus we put a wall segment right below the beetle. &
        We build a wall until it surrounds some area of $M_7$.&
        We draw the ray from the center to the bottom. The ray intersects two wall segments.&
        The beetle gets around the wall. Afterwards, it moves down and reaches the boundary.\\
        \hline
    \end{tabular}
    \end{center}
    \end{table}

\section{Acknowledgements}
    The author is grateful to his research advisor M.~Skopenkov for setting the problems and constant attention to this work; to K.~Izyurov, A.~Magazinov, and M.~Khristoforov for reading this work and sending valuable comments; to K.~Izyurov and A.~Magazinov for telling the proof of Theorem~\ref{sides}.

    {}
\end{document}


\author{И.\,В.~Новиков}
\title{\textbf{Протекание трёх жидкостей на шестиугольной решетке}}
\date{}
\maketitle
\setcounter{section}{0}
\begin{abstract}
    В данной работе рассматривается естественное обобщение перколяции: протекание трёх взаимосвязанных жидкостей на шестиугольной решётке. К.~Изъюров и А.~Магазинов доказали, что протекания разных жидкостей между противоположными сторонами фиксированного шестиугольника становятся независимыми в совокупности при измельчении решётки. Мы подробно изложим это доказательство для неспециалистов с незначительными упрощениями. Также мы сформулируем несколько новых близких гипотез на основе численных экспериментов.
\end{abstract}

\keywords{Преобразование Фурье-Уэлша, модель Поттса, протекание, шестиугольная решётка, булевы функции.}

\tableofcontents

\section{Введение}
    В данной работе рассматривается естественное обобщение перколяции: протекание трёх взаимосвязанных жидкостей на шестиугольной решётке. Мы докажем (см. теорему \ref{russides}), что протекания разных жидкостей между противоположными сторонами фиксированного шестиугольника становятся независимыми в совокупности при измельчении решётки. Данное утверждение в качестве гипотезы на основе численных экспериментов предложил М. Скопенков, а затем приблизительно в одно и то же время независимо доказали К.~Изъюров и А.~Магазинов (частное сообщение) с помощью разложения Фурье-Уэлша и теоремы Кестена. В данной работе это доказательство подробно излагается для неспециалистов, с незначительными упрощениями. Помимо доказательства теоремы, мы сформулируем несколько новых близких гипотез на основе численных экспериментов.

    \textbf{План работы.} В \S2 вводятся основные понятия, формулируется основная теорема \ref{russides} и несколько гипотез. В \S3 вводится разложение Фурье-Уэлша. В \S4 при помощи этого разложения проводится доказательство теоремы \ref{russides} из \S2. В \S5 описаны результаты численных экспериментов, связанных с гипотезами из \S2.

\section{Формулировка основного результата}
    Для любого $n>1$ рассмотрим шестиугольную решётку со стороной клетки $\frac{1}{n}$. Пусть $P$ — правильный шестиугольник со стороной~$1$, центр которого расположен в центре некоторой клетки $O$. Обозначим через $M_n$ множество всех клеток, принадлежащих $P$, включая границу. Рассмотрим вероятностное пространство $\Omega$, состоящее из всех раскрасок клеток множества $M_n$ в $4$ цвета, обозначаемых числами $0$, $1$, $2$, $3$, с мерой $P(B)=|B|/4^{|M_n|}$ для любого $B \subset \Omega$. Оно называется \textit{моделью Поттса с 4 состояниями при бесконечной температуре.}

    \begin{definition}
    \label{rus4 color sides}
        Зададим число $k=1$, $2$ или $3$ и некоторую раскраску клеток множества~$M_n$ в $4$ цвета. Будем говорить, что \textit{жидкость $k$ протекает между двумя множествами $A, B \subset~M_n$}, если некоторую клетку множества $A$ можно соединить c некоторой клеткой множества $B$ цепочкой клеток, соседних по стороне, в которой каждая клетка имеет цвет $0$ или $k$, включая начальную клетку множества $A$ и конечную клетку множества $B$.
    \end{definition}
\subsection{Протекание между сторонами}
    Будем говорить, что \textit{клетка $x \in M_n$ принадлежит стороне $A$ шестиугольника $P$}, если $x$ является граничной и $A$ является ближайшей к $x$. Клетка может принадлежать более, чем одной стороне (если окажется, что ближайших сторон несколько). В дальнейшем множество клеток, принадлежащих стороне $A$, будем также называть \textit{стороной} $A$.

    Пронумеруем стороны шестиугольника $P$ числами от $1$ до $6$ против часовой стрелки. Для $k=1, 2$ или $3$ обозначим через $B_{k,n} \subset \Omega$ множество таких раскрасок, что жидкость $k$ протекает между сторонами $k$ и $k+3$ шестиугольника~$P$. Легко видеть, что события $B_{1,n}, B_{2,n}, B_{3,n}$ попарно независимы: $P(B_{1,n} \cap B_{2,n}) - P(B_{1,n})P(B_{2,n})=0$.

    \begin{theorem}[К. Изъюров, А. Магазинов, 2018]
    \label{russides}
        $$\displaystyle{\lim_{n \to +\infty}} \left[P(B_{1,n} \cap B_{2,n} \cap B_{3,n}) - P(B_{1,n}) P(B_{2,n}) P(B_{3,n})\right] = 0.$$
    \end{theorem}

    \begin{remark}
    Теорема \ref{russides} верна в значительно большей общности: например, правильный шестиугольник $P$ можно заменить на произвольный многоугольник, а противоположные стороны --- на любые пары сторон. Однако для простоты доказательства мы будем рассматривать именно протекание между противоположными сторонами правильного шестиугольника.
    \end{remark}

    Неформально, теорема \ref{russides} означает, что протекания различных жидкостей между сторонами становятся независимыми в совокупности при $n \to \infty$. Мы рассматриваем разность, а не отношение, чтобы не пришлось доказывать отделимость $P(B_{1,n}) P(B_{2,n}) P(B_{3,n})$ от нуля.

    Сформулируем дополнительно гипотезу.
    \begin{hypothesis}
    \label{rushyp3}
        $P(B_{1,n} \cap B_{2,n} \cap B_{3,n}) \geqslant P(B_{1,n}) P(B_{2,n}) P(B_{3,n})$ для всех $n$.
    \end{hypothesis}

      С учётом попарной независимости событий $B_{1,n}, B_{2,n}, B_{3,n}$ эта гипотеза означает, что протекания различных жидкостей положительно скоррелированы, т.е. $$P(B_{1,n}|B_{2,n} \cap B_{3,n}) \geqslant B_{1,n}.$$

\subsection{Протекание из центра}

    \begin{definition}
    \label{rus4 color center}
        Зададим число $k=1$, $2$ или $3$ и некоторую раскраску клеток множества $M_n$ в $4$ цвета. Будем говорить, что \textit{жидкость $k$ протекает из клетки $x \in M_n$ до множества $A \in M_n$}, если $x$ можно соединить c некоторой клеткой множества $A$ цепочкой клеток, соседних по стороне, в которой каждая клетка имеет цвет $0$ или $k$, \textit{не} включая начальную клетку $x$ и включая конечную клетку множества $A$.
    \end{definition}

    Отметим, что в определении \ref{rus4 color center}, в отличие от определения \ref{rus4 color sides}, цвет стартовой клетки $x$ не имеет значения. Это удобно, так как иначе в формулировках гипотез \ref{rushyp1} и \ref{rushyp2} ниже появился бы дополнительный множитель 2 в правых частях.

    Для $k = 1, 2$ или $3$ обозначим через $A_{k,n} \subset \Omega$ множество таких раскрасок, что жидкость $k$ протекает из клетки $O$ до множества граничных клеток множества $M_n$.

    Сформулируем 2 гипотезы.

    \begin{hypothesis}
    \label{rushyp1}
        $P(A_{1,n} \cap A_{2,n} \cap A_{3,n}) \geqslant P(A_{1,n}) P(A_{2,n}) P(A_{3,n})$ для всех $n$.\\
    \end{hypothesis}

    \begin{hypothesis}
    \label{rushyp2}
        $\displaystyle{\lim_{n \to +\infty}} \frac{P(A_{1,n} \cap A_{2,n} \cap A_{3,n})}{P(A_{1,n}) P(A_{2,n}) P(A_{3,n})} > 1$.
    \end{hypothesis}


    Дополнительно приведём ещё одну гипотезу, по сути обобщающую гипотезы \ref{rushyp3} и \ref{rushyp1}.
    \begin{hypothesis}
    \label{rushyp4}
        Пусть $A_1, A_2, A_3, B_1, B_2, B_3 \subset M_n$. Обозначим $U_k = $ \{жидкость $k$ протекает между множествами $A_k$ и~$B_k$\}. Тогда
        $$P(U_{1} \cap U_{2} \cap U_{3}) \geqslant P(U_{1}) P(U_{2}) P(U_{3}).$$
    \end{hypothesis}
    В случае истинности гипотезы \ref{rushyp4} выполнено, что $P(U_{1}|U_{2} \cap U_{3}) \geqslant P(U_{1})$.
\subsection{Пример}
    \begin{wrapfigure}{o}{0.2\linewidth} 
    \center{\includegraphics[width=1\linewidth]{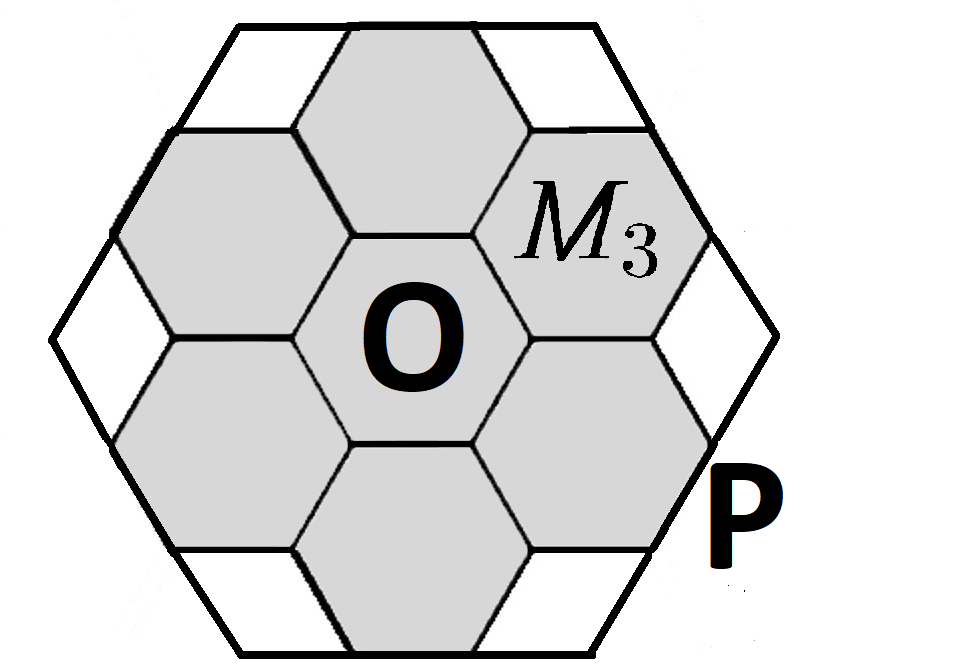}}
    \caption{}
    \end{wrapfigure}
    Проиллюстрируем на примере введённые выше понятия.

    \begin{example}
        Пусть $n=3$. На рисунке 1 изображён правильный шестиугольник~$P$, расположенный на шестиугольной решётке со стороной клетки $\frac{1}{3}$. Множество $M_3$ содержит 7 серых клеток. Легко видеть, что
        $$P(A_{1,3}) = P(A_{2,3}) = P(A_{3,3}) = 1 - \left(\frac{1}{2}\right)^6 = 0,984375;$$
        $$P(A_{1,3} \cap A_{2,3} \cap A_{3,3}) = 1 - 3\left(\frac{1}{2}\right)^6 + 3 \left(\frac{1}{2}\right)^{12} =0,953857421875;$$
        $$P(B_{1,3}) = P(B_{2,3}) = P(B_{3,3}) = \frac{1}{4}\left(\frac{1}{2} + \frac{1}{2} \cdot \frac{7}{16}\right)= \frac{23}{128} = 0,1796875;$$
        $$P(B_{1,3} \cap B_{2,3} \cap B_{3,3}) =
        \frac{1}{4} \cdot \left(\frac{1}{2}\right)^6 + \frac{3}{4} \cdot \left( 9 \left(\frac{1}{4}\right)^6 + 2 \left(\frac{1}{2}\right)^2 \left(\frac{1}{4}\right)^4\right) = 0,00701904296875.$$

        Тогда
        $$\frac{P(A_{1,3} \cap A_{2,3} \cap A_{3,3})}{P(A_{1,3}) P(A_{2,3}) P(A_{3,3})} = \frac{250048}{250047} > 1
        \text{ и }$$
        $$P(B_{1,3} \cap B_{2,3} \cap B_{3,3}) - P(B_{1,3}) P(B_{2,3}) P(B_{3,3}) \approx 0,01.$$
        что согласуется с гипотезами \ref{rushyp3} и \ref{rushyp1}, а также показывает, что как события $A_{1,3}, A_{2,3}, A_{3,3}$, так и события $B_{1,3}, B_{2,3}, B_{3,3}$ \textit{зависимы} в совокупности.
    \end{example}
\subsection{Теорема Кестена}

    Введённые нами понятия естественно обобщают классическую модель перколяции, когда клетки решётки красятся в 2 цвета. Сформулируем известный результат о перколяции, который пригодится нам для доказательства теоремы \ref{russides}.

    \begin{definition}[ср. с определением \ref{rus4 color center}]
    \label{rus2 color center}
        Пусть дана раскраска клеток некоторого конечного множества $M$ в $2$ цвета $\pm 1$. Будем говорить, что для этой раскраски \textit{существует протекание из клетки $x \in M$ до множества $A \subset M$}, если $x$ можно соединить c некоторой клеткой множества $A$ цепочкой клеток, соседних по стороне, в которой каждая клетка имеет цвет $+1$, \textit{не} включая $x$ и включая конечную клетку множества~$A$. Введём на множестве раскрасок множества $M$ в 2 цвета меру $P(B)~=~|B|/2^m$.
    \end{definition}

    \begin{theorem}(Г. Кестен, 1982, ср. с \cite[Theorem~9.6]{rusAriel Yadin}).
        \label{ruskesten lite}
        При $n\to\infty$ вероятность существования протекания из клетки $O$ до границы множества $M_n$ стремится к $0$.
    \end{theorem}
\section{Разложение Фурье-Уэлша}
\label{russection fourier}
    В этом разделе мы введём понятие разложения Фурье-Уэлша функций и рассмотрим некоторые свойства коэффициентов этого разложения. Содержание \S\S \ref{rusFourier subsection 1}-\ref{rusFourier subsection 3} заимствовано из книги \cite[\S\S 1.2-1.4]{rusboolean functions}. Лемма \ref{rusformula probability} из подраздела 3.4 аналогична результатам из \cite[\S 2.2]{rusboolean functions}. Лемма \ref{rusinequality} из подраздела 3.4 приведена в \cite[\S 3.6]{rusboolean functions} в качестве упражнения.

\begin{agreement}
    Всюду далее в этой статье на множестве $\{-1,+1\}^m$, где $m~\in~\mathbb{N}$, фиксирована мера $P(B) = |B|/2^m$ для любого $B \subset \{-1,+1\}^m$. Функции $f \colon \{-1,+1\}^m \to \mathbb{R}$ рассматриваются как случайные величины на $\{-1,+1\}^m$.
\end{agreement}

\subsection{Определение}
\label{rusFourier subsection 1}

    \begin{definition}
    \label{rusFourier}
    Пусть $m \in \mathbb{N}$. \textit{Разложением Фурье-Уэлша функции $$f \colon \{-1,+1\}^m \to \mathbb{R}$$} называется её представление в виде суммы $$f(x_1, \ldots, x_m) = \sum_{S \subset \{1, 2, \ldots, m\}} \widehat{f_{S}} \prod_{i \in S} x_i,$$ где $\widehat{f_{S}}$ — некоторые вещественные числа, называемые \textit{коэффициентами Фурье-Уэлша}. Если $S = \emptyset$, то считаем по определению, что $\prod_{i \in S} x_i = 1$.
    \end{definition}

    \begin{remark}
         Далее  $\widehat{f_{S}}$ всегда будет обозначать коэффициент при $\prod_{i \in S} x_i$ в разложении Фурье-Уэлша функции $f \colon \{-1,+1\}^m \to \mathbb{R}$.
    \end{remark}

    \begin{theorem}
        Для любой функции $f \colon \{-1,+1\}^m \to \mathbb{R}$ разложение Фурье-Уэлша существует и единственно.
    \end{theorem}
    \begin{proof}
        Докажем существование разложения. Для каждой точки $a =(a_1, a_2, \ldots, a_m) \in \mathbb{R}^m$, где $a_i = \pm 1$, рассмотрим функцию $$1_{a}(x_1, x_2, \ldots, x_m) = \left(\frac{1+a_1 x_1}{2}\right)\left(\frac{1+a_2 x_2}{2}\right) \cdots \left(\frac{1+a_m x_m}{2}\right),$$ определённую при $x_i = \pm 1$. Заметим, что
        \begin{equation*}
            1_{a}(x)=
            \begin{cases}
                $1$, &\text{если } x = a;\\
                $0$, &\text{иначе.}
            \end{cases}
        \end{equation*}
        Значит, произвольную функцию $f:\{-1, +1\}^m \rightarrow \mathbb{R}$ можно записать в виде $$f(x)= {\sum_{a \in \{-1,+1\}^m}} f(a) 1_a(x).$$ Раскрыв скобки и сгруппировав слагаемые, получаем требуемое разложение.

        Докажем теперь единственность коэффициентов. Множество всех функций $f:\{-1, +1\}^m~\rightarrow~\mathbb{R}$ образуют \mbox{$2^m$-мерное} векторное пространство. Но мономов вида $\prod_{i \in S} x_i$, где $S~\subset~\{1,2, \ldots, m \}$, как раз $2^m$. Поскольку любая функция является линейной комбинацией таких мономов, то эти мономы образуют базис. Из этого получаем единственность коэффициентов Фурье-Уэлша.
    \end{proof}

\subsection{Формула для коэффициентов}
\label{rusFourier subsection 2}
    В этом разделе приводится формула для коэффициентов Фурье-Уэлша через матожидание некоторой случайной величины (лемма \ref{rusformula} ниже).

    \begin{Symbol}
        Пусть $S \subset \{1, 2, \ldots, m\}$. Обозначим через $\sigma_{S}$ функцию \\ $\sigma_S~\colon~\{-1,+1\}^m~\to~\{-1,+1\}$, заданную формулой
        $\sigma_{S}(x_1, \ldots, x_m) = \prod_{i \in S} x_i$.
    \end{Symbol}

    \begin{lemma}
        \label{rusmultiplication of sigmas}
        Пусть $S,T \subset \{1, 2, \ldots, m\}$. Тогда $\sigma_{S}\sigma_{T}=\sigma_{S \triangle T}$.
    \end{lemma}

    \begin{proof}
        $\sigma_{S}\sigma_{T} = \displaystyle{\prod_{i \in S} x_i \prod_{j \in T} x_i = \prod_{i \in S \cap T} x_i^2 \prod_{j \in S \triangle T} x_i} = \prod_{j \in S \triangle T} x_i = \sigma_{S \triangle T}.$
    \end{proof}

    \begin{lemma}
    \label{rusindependency of expectaion}
        Случайные величины $\sigma_{\{1\}}, \sigma_{\{2\}}, \ldots, \sigma_{\{m\}}$ имеют нулевое математическое ожидание и независимы в совокупности.
    \end{lemma}

    Доказательство этой леммы очевидно.

    \begin{lemma}
        \label{rusexpecation of sigmas}
        Пусть $S \subset \{1, 2, \ldots, m\}$. Тогда
        \begin{equation*}
            \mathbb{E}\sigma_{S}=
            \begin{cases}
                $1$, &\text{если } S = \emptyset;\\
                $0$, &\text{иначе.}
            \end{cases}
        \end{equation*}
    \end{lemma}
    \begin{proof}
        Если $S = \emptyset$, то $\mathbb{E}\sigma_{\emptyset} = \mathbb{E}1 = 1$.

        Если $S \neq \emptyset$, то с учётом леммы \ref{rusindependency of expectaion} имеем
        $$\mathbb{E}(\sigma_{S}) = \mathbb{E}\prod_{i \in S} x_i = \prod_{i \in S} \mathbb{E}x_i = 0.$$
    \end{proof}

    \begin{lemma}[Формула для коэффициентов Фурье-Уэлша]
        \label{rusformula}
        Пусть $T \subset \{1, 2, \ldots, m\}$. Тогда для любой функции $f \colon \{-1,+1\}^m \to \mathbb{R}$ справедлива формула
        $$\widehat{f_{T}} = \mathbb{E}(f \sigma_{T}).$$
    \end{lemma}

    \begin{proof}
        С учётом определения \ref{rusFourier}, лемм \ref{rusmultiplication of sigmas} и \ref{rusexpecation of sigmas} имеем
        $$\mathbb{E}(f \sigma_{T}) =
        \mathbb{E}\left[\sigma_{T} \cdot \displaystyle{\sum_{S \subset \{1, 2, \ldots, m\}}} \widehat{f_{S}} \cdot \sigma_{S}\right]
        =\mathbb{E}\left[\displaystyle{\sum_{S \subset \{1, 2, \ldots, m\}}} \widehat{f_{S}} \cdot \sigma_{S \triangle T}\right] = \widehat{f_{T}}.$$
    \end{proof}
\subsection{Дисперсия}
\label{rusFourier subsection 3}
    Нам будет полезна аналогичная формула для дисперсии (следствие $\ref{rusVar}$ ниже). Для начала докажем простой факт.

    \begin{lemma}
    \label{rusVar<=1}
        Для любой функции $f \colon \{-1, +1\}^m \to \{0, 1\}$ справедливо, что
        $$0 \leqslant Var(f) \leqslant 1.$$
    \end{lemma}

    \begin{proof}Очевидно, что $\mathbb{E}f \in [0,1]$. Поэтому $$Var(f) = \mathbb{E}(f^2)-(\mathbb{E}f)^2 = \mathbb{E}f-(\mathbb{E}f)^2 \in [0,1].$$
    \end{proof}

    \begin{lemma}
    \label{rusParsevall}
        Для любой функции $f \colon \{-1, +1\}^m \to \mathbb{R}$ справедлива формула
        $$\mathbb{E}(f^2) = \sum_{S \subset \{1, \ldots, m\}} \widehat{f_S}^2.$$
    \end{lemma}

    \begin{proof}
        С учётом лемм \ref{rusmultiplication of sigmas} и \ref{rusexpecation of sigmas} имеем
        \begin{align*}
        \mathbb{E}(f^2) &= \mathbb{E}\left[\sum_{S \subset \{1, 2, \ldots, m\}} \widehat{f_{S}} \sigma_S\right]^2 =
        \mathbb{E}\left(\sum_{S \subset \{1, 2, \ldots, m\}} \widehat{f_S}^2 \sigma_S^2 + 2\underset{S \neq T}{\sum_{S, T \subset \{1, 2, \ldots, m\}}} \widehat{f_{S}}\widehat{f_{T}} \sigma_{S \triangle T}\right) = \sum_{S \subset \{1, \ldots, m\}} \widehat{f_S}^2.
        \end{align*}
    \end{proof}

    \begin{corollary}
    \label{rusVar}
        Для любой функции $f \colon \{-1, +1\}^m \to \mathbb{R}$ справедлива формула
        $$Var(f) = \underset{S \neq \emptyset}{\sum_{S \subset \{1, \ldots, m\}}}\widehat{f_S}^2.$$
    \end{corollary}

    \begin{proof}
         Так как $\sigma_{\emptyset} = 1$, то с учётом лемм \ref{rusformula} и \ref{rusParsevall} получаем
        $$Var(f)=\mathbb{E}(f^2)-(\mathbb{E}f)^2 = \underset{S \neq \emptyset}{\sum_{S \subset \{1, \ldots, m\}}}\widehat{f_S}^2.$$
    \end{proof}

\subsection{Возрастающие булевы функции}
    В этом параграфе мы рассматриваем только функции со значениями в $\{0, 1\}$.

    \begin{definition}
    \label{rusincreasing}
        Пусть $a= (a_1, a_2, \ldots, a_m), b = (b_1, b_2, \ldots, b_m) \in \{-1, 1\}^m$. Мы пишем $a \leqslant b$, если для всех
        $i=1, \ldots, m$ верно, что $a_i \leqslant b_i$. Функция $f \colon \{-1, +1\}^m \to \{0, 1\}$ называется \textit{возрастающей}, если из неравенства
        $a\leqslant b$ следует $f(a)~\leqslant~f(b).$
    \end{definition}

    \begin{definition}
    \label{rusi-crit}
        Для данной функции $f \colon \{-1, +1\}^m \to \{0, 1\}$ назовём точку $(x_1, \ldots, x_m) \in \{-1, +1\}^m$ \textit{ $i$-критической}, если
        $$x_i =1 \text{ и } f(x_1,\ldots,x_{i-1}, +1, x_{i+1}, \ldots,x_m) \neq f(x_1,\ldots,x_{i-1}, -1, x_{i+1},\ldots,x_m).$$
    \end{definition}

    \begin{lemma}
        \label{rusformula probability}
        Для любой возрастающей функции $f \colon \{-1, +1\}^m \to \{0, 1\}$ и любого
        $i \in \{1, \ldots, m\}$ справедлива формула
        $$\widehat{f_{\{i\}}} = P(\{x \in \{-1, +1\}^m : x\text{ является }i\text{-критической}\}).$$
    \end{lemma}

    \begin{proof}
        Введём обозначения:
        \begin{align*}
        B_+ &= \{x \in \{-1, +1\}^m : f(x) = 1, x_i=+1 \};\\
        B_- &= \{x \in \{-1, +1\}^m : f(x) = 1, x_i=-1 \}.
        \end{align*}

        Рассмотрим вложение $g \colon B_- \hookrightarrow{} B_+$, определённое для каждого $x \in B_-$ равенством
        $$g(x_1, \ldots, x_{i-1},-1,x_{i+1}, \ldots, x_m) = (x_1, \ldots, x_{i-1},+1,x_{i+1}, \ldots, x_m).$$ Так как функция $f$ возрастающая, то отображение определено корректно. Инъективность очевидна. По лемме \ref{rusformula} имеем
        \begin{align*}
        \widehat{f_{\{i\}}} &= \mathbb{E}(f(x) x_i) = \frac{|B_+|-|B_-|}{2^m} = \frac{|g(B_-)|+|B_+ \setminus g(B_-)|-|B_-|}{2^m} = \frac{|B_+ \setminus g(B_-)|}{2^m} = \\
        &= \frac{|\{x \in \{-1, +1\}^m : x \text{ является }i\text{-критической}\}|}{2^m} =\\
        &= P(\{x \in \{-1, +1\}^m : x \text{ является }i\text{-критической}\}).
        \end{align*}
    \end{proof}

    \begin{lemma}
    \label{rusinequality}
        Пусть $S \subset \{1, 2, \ldots, m\}, S \neq \emptyset$. Если функция $f~\colon~\{-1, +1\}^m~\to~\{0, 1\}$ возрастает, то для любого $i \in S$ выполнено $|\widehat{f_S}| \leqslant \widehat{f_{\{i\}}}$.
    \end{lemma}

    \begin{proof}
        Зафиксируем $i \in S$. Введём обозначения:
        \begin{align*}
        C_{++} &= \{x \in \{-1, +1\}^m : f(x) = 1, x_i=+1, \sigma_{S}(x) = +1\};\\
        C_{+-} &= \{x \in \{-1, +1\}^m : f(x) = 1, x_i=+1, \sigma_{S}(x) = -1\};\\
        C_{-+} &= \{x \in \{-1, +1\}^m : f(x) = 1, x_i=-1, \sigma_{S}(x) = +1\};\\
        C_{--} &= \{x \in \{-1, +1\}^m : f(x) = 1, x_i=-1, \sigma_{S}(x) = -1\}.
        \end{align*}

        Заметим, что $\{x \in \{-1, +1\}^m : f(x) = 1\} = C_{++} \sqcup C_{+-} \sqcup C_{-+} \sqcup C_{--}$.

        По лемме \ref{rusformula} имеем
        \begin{equation}
        \label{rusf_i=}
        \widehat{f_{\{i\}}} = \mathbb{E}(f x_i ) = \frac{|C_{++} \sqcup C_{+-}| - |C_{--} \sqcup C_{-+}|}{2^m};
        \end{equation}

        \begin{equation}
        \label{rusf_S=}
        \widehat{f_S} = \mathbb{E}(f \sigma_{S})= \frac{|C_{++} \sqcup C_{-+}| - |C_{+-} \sqcup C_{--}|}{2^m}.
        \end{equation}

        Построим вложение $h \colon C_{-+} \hookrightarrow C_{+-}$, определённое для каждого $x \in C_{-+}$ равенством
        $$h(x_1, \ldots, x_{i-1},-1,x_{i+1}, \ldots, x_m) = (x_1, \ldots, x_{i-1},+1,x_{i+1}, \ldots, x_m).$$
        Так как функция $f$ возрастает и $i\in S$, то отображение определено корректно. Инъективность очевидна. Таким образом,
        $|C_{++} \sqcup C_{-+}| \leqslant |C_{++} \sqcup C_{+-}|$ и $|C_{+-}~\sqcup~C_{--}|~\geqslant~|C_{--} \sqcup C_{-+}|$. Сопоставляя (\ref{rusf_i=}) и (\ref{rusf_S=}) получаем, что $\widehat{f_S} \leqslant \widehat{f_{\{i\}}}$.

        Аналогичным образом доказывается, что $\widehat{f_S} \geqslant - \widehat{f_{\{i\}}}$.
    \end{proof}

    Из леммы \ref{rusinequality} получаем очевидное следствие.
    \begin{corollary}
    \label{rusmaximum}
        Пусть функция $f~\colon~\{-1, +1\}^m~\to~\{0, 1\}$ возрастает. Тогда наибольший коэффициент Фурье-Уэлша, кроме $\widehat{f_{\emptyset}}$, — это один из $\widehat{f_{\{1\}}}, \ldots, \widehat{f_{\{m\}}}$.
    \end{corollary}

\section{Доказательство основной теоремы}
\subsection{Переформулировка в терминах коэффициентов Фурье-Уэлша}
\label{rusperef}
    Сопоставим каждой раскраске множества $M_n$ в 4 цвета три раскраски этого множества в два цвета, обозначаемых $\pm 1$. Таблица \ref{russop} показывает, какой цвет сопоставляется каждой клетке.

    \begin{table}[h]
    \caption{Построение новых раскрасок}
    \vspace{-6mm}
    \begin{center}
    \label{russop}
    \begin{tabular}{|c|c|c|c|}
        \hline
        Раскраска в 4 цвета & Раскраска 1 & Раскраска 2 & Раскраска 3\\
        \hline
        0 & $+1$ & $+1$ & $+1$\\
        1 & $+1$ & $-1$ & $-1$\\
        2 & $-1$ & $+1$ & $-1$\\
        3 & $-1$ & $-1$ & $+1$\\
        \hline
    \end{tabular}
    \end{center}
    \vspace{-7mm}
    \end{table}

    \begin{remark}
    \label{russved}
        Существует биекция между раскрасками множества $M_n$ $$\{\text{раскраски }\text{в 4 цвета } 0, 1, 2, 3\} \to \{\text{раскраски }\text{в 2 цвета } \pm1\}^2,$$
        сопоставляющая каждой раскраске в 4 цвета пару (раскраска 1, раскраска 2). Эта биекция сохраняет меру на множествах раскрасок.
    \end{remark}

    \begin{remark}
    \label{rus3th coloring}
        Для того, чтобы узнать цвет клетки в раскраске 3, достаточно перемножить цвета этой клетки в раскрасках 1 и 2.
    \end{remark}

    \begin{definition}[ср. с определением \ref{rus4 color sides}]
    \label{rus2 color sides}
         Пусть дана раскраска клеток множества $M_n$ в $2$ цвета $\pm1$. Будем говорить, что \textit{в $M_n$ существует протекание между двумя сторонами $A$ и $B$ шестиугольника $P$}, если некоторую клетку стороны $A$ можно соединить c некоторой клеткой стороны $B$ цепочкой клеток, соседних по стороне, в которой каждая клетка имеет цвет $+1$, включая начальную клетку на стороне $A$ и конечную клетку на стороне $B$.
    \end{definition}

    \begin{remark}
    \label{russtor}
        Жидкость $k$ протекает между двумя сторонами $A$ и $B$ шестиугольника в изначальной раскраске в 4 цвета тогда и только тогда, когда  в раскраске~$k$ существует протекание между двумя сторонами $A$ и $B$.
    \end{remark}

    Пусть в множестве $M_n$ всего $m$ клеток. Занумеруем их. С этого момента введём обозначения для раскрасок множества $M_n$ в 2 либо в 4 цвета. 

    \begin{Symbol}[раскраски в 2 цвета]
        Каждую раскраску в 2 цвета будем далее обозначать как точку $x = (x_1, \ldots, x_m)~\in~\{-1,+1\}^m$.
    \end{Symbol}

    \begin{Symbol}[раскраски в 4 цвета]
        Отождествим раскраски в 4 цвета с парами раскрасок в 2 цвета с помощью биекции из замечания \ref{russved}. Обозначим через $x_i$ и $y_i$ цвет в $i$-ой клетке в раскрасках 1 и 2 соответственно. Каждую раскраску в 4 цвета далее будем обозначать как точку
        $$(x,y) = (x_1, \ldots, x_m, y_1, \ldots, y_m)~\in~\{-1,+1\}^m~\times~\{-1,+1\}^m.$$
    \end{Symbol}

    Для $k=1, 2, 3$ рассмотрим следующие функции от раскраски $x \in \{-1,+1\}^m$:
    \begin{equation*}
    f^k(x)=
    \begin{cases}
        \multirow{2}{*}{1,} &\text{если в раскраске $x$ существует протекание} \\
                            &\text{между сторонами $k$ и $k+3$ шестиугольника $P$;}\\
        0, &\text{иначе.}
    \end{cases}
    \end{equation*}

    \begin{Symbol}
        Пусть $S \subset \{1, 2, \ldots, m\}$. Обозначим через $\sigma_{S}$ и $\tau_S$ функции \\ $\{-1,+1\}^m~\times~\{-1,+1\}^m~\to~\{-1,+1\}$, заданные формулами
        \begin{align*}
        \sigma_{S}(x, y) &= \sigma_{S}(x_1, \ldots, x_m, y_1, \ldots, y_m) = \prod_{i \in S} x_i,\\
        \tau_{S}(x, y) &= \tau_{S}(x_1, \ldots, x_m, y_1, \ldots, y_m) = \prod_{i \in S} y_i.
        \end{align*}
    \end{Symbol}

    Сформулируем несколько свойств функций $\sigma_S$ и $\tau_S$. Доказательства аналогичны доказательствам лемм \ref{rusmultiplication of sigmas}-\ref{rusexpecation of sigmas}.

    \begin{lemma}
    \label{russigma tau}
        Пусть $S,T \subset \{1, 2, \ldots, m\}$. Тогда:\\
        a) $\sigma_{S}\sigma_{T}=\sigma_{S \triangle T}$ и $\tau_{S}\tau_{T}=\tau_{S \triangle T}$.\\
        b) Случайные величины $\sigma_{\{1\}}, \sigma_{\{2\}}, \ldots, \sigma_{\{m\}}, \tau_{\{1\}}, \tau_{\{2\}}, \ldots, \tau_{\{m\}}$ (на пространстве $\{-1,+1\}^m\times\{-1,+1\}^m)$ имеют нулевое математическое ожидание и независимы в совокупности.\\
        c)
            $\mathbb{E}(\sigma_{S}\tau_{T})=
            \begin{cases}
                $1$, &\text{если } S = T = \emptyset;\\
                $0$, &\text{иначе.}
            \end{cases}$
    \end{lemma}

    Сформулируем основную лемму данного подраздела. 

    \begin{lemma}
    \label{rusformulae for sides probability}
     $P(B_{k,n}) = \widehat{f^k_{\emptyset}}$ для любого $k = 1, 2, 3$, и $P(B_{1,n} \cap B_{2,n} \cap B_{3,n}) =~\displaystyle{\sum_{S \subset \{1, \ldots, m\}}~\widehat{f^1_S} \widehat{f^2_S} \widehat{f^3_S}}.$
     \end{lemma}

    \begin{proof}
        Матожидание в этом доказательстве берётся в пространстве $\{-1,+1\}^m$ в первых трёх формулах и в пространстве $\{-1,+1\}^m\times\{-1,+1\}^m$ в последней формуле. Мы обозначаем $f^3(x y) = f^3(x_1 y_1, \ldots, x_m y_m)$.

        С учётом определения \ref{rusFourier}, лемм \ref{rusformula} и \ref{russigma tau}, а также замечаний \ref{russtor}, \ref{russved} и \ref{rus3th coloring} имеем
        \begin{align*}
            P(B_{1,n}) &= \frac{|\{(x,y) \in \{-1,+1\}^m\times\{-1,+1\}^m : f^1(x) = 1\}|}{4^m} = \\
            &= \frac{2^m |\{x \in \{-1,+1\}^m : f^1(x) = 1\}|}{4^m} = \mathbb{E} f^1 = \widehat{f^1_{\emptyset}};\\
            P(B_{2,n}) &= \frac{|\{(x,y) \in \{-1,+1\}^m\times\{-1,+1\}^m : f^2(y) = 1\}|}{4^m} = \\
            &= \frac{2^m |\{y \in \{-1,+1\}^m : f^1(y) = 1\}|}{4^m} = \mathbb{E} f^2 = \widehat{f^2_{\emptyset}};\\
            P(B_{3,n}) &= \frac{|\{(x,y) \in \{-1,+1\}^m\times\{-1,+1\}^m : f^3(x y) = 1\}|}{4^m} = \\
            &= \frac{2^m |\{z \in \{-1,+1\}^m : f^3(z) = 1\}|}{4^m} = \mathbb{E} f^3 = \widehat{f^3_{\emptyset}}.
        \end{align*}
        (Обращаем внимание, что доказательство третьей формулы немного отличается от первых двух).

        \begin{align*}
        &P(B_{1,n} \cap B_{2,n} \cap B_{3,n}) = \frac{|\{(x,y)\in\{-1,+1\}^m\times\{-1,+1\}^m : f^1(x) f^2(y) f^3(x y) = 1\}|}{4^m} = \\
        &= \mathbb{E} \left[f^1(x) f^2(y) f^3(x y)\right] = \mathbb{E}\left[\sum_{S \subset \{1, \ldots, m\}} \widehat{f^1_{S}} \sigma_S \sum_{S \subset \{1, \ldots, m\}} \widehat{f^2_{S}} \tau_S \sum_{S \subset \{1, \ldots, m\}} \widehat{f^3_{S}} \sigma_S \tau_S\right] =\\
        &= \mathbb{E}\left[\sum_{S \subset \{1, \ldots, m\}}~\widehat{f^1_S} \widehat{f^2_S} \widehat{f^3_S} \sigma_S^2 \tau_S^2 + \underset{\text{не все }S, T, R \text{ равны}}{\sum_{S, T, R \subset \{1, \ldots, m\}}} \widehat{f^1_S} \widehat{f^2_T} \widehat{f^3_R} \sigma_{S \triangle R}  \tau_{T \triangle R}\right] =
        \sum_{S \subset \{1, \ldots, m\}}~\widehat{f^1_S} \widehat{f^2_S} \widehat{f^3_S}.
        \end{align*}
    \end{proof}

    \begin{corollary}
    \label{russides refolmulation}
       $P(B_{1,n} \cap B_{2,n} \cap B_{3,n}) -  P(B_{1,n}) P(B_{2,n}) P(B_{3,n}) = \displaystyle{\underset{S \neq \emptyset}{\sum_{S \subset \{1, \ldots, m\}}}}~\widehat{f^1_S} \widehat{f^2_S} \widehat{f^3_S}.$
    \end{corollary}

    Значит, для доказательства теоремы \ref{russides} достаточно доказать, что $$\lim_{n \to \infty}\underset{S \neq \emptyset}{\sum_{S \subset \{1, \ldots, m\}}}~\widehat{f^1_S} \widehat{f^2_S} \widehat{f^3_S} = 0.$$

\subsection{Критические раскраски}

    Отметим, что функции $f^1, f^2, f^3$ возрастают, что очевидно следует из определения этих функций и определения \ref{rusincreasing}. Значит, для этих функций выполнены все леммы из раздела \ref{russection fourier}, чем мы сейчас будем пользоваться.

    \begin{lemma}
    \label{rusineq}
        $$\left|\underset{S \neq \emptyset}{\sum_{S \subset \{1, \ldots, m\}}}~\widehat{f^1_S} \widehat{f^2_S} \widehat{f^3_S}\right| \leqslant \max_{i \in \{1, \ldots, m \}} P(\{x \in \{-1,+1\}^m :x \text{ является }i\text{-критической для }f^1\}).$$
    \end{lemma}

    \begin{proof} Из неравенства Коши-Буняковского-Шварца, следствия \ref{rusVar}, а также лемм \ref{rusVar<=1}, \ref{rusinequality} и \ref{rusformula probability} имеем

        \begin{align*}
        \left|\underset{S \neq \emptyset}{\sum_{S \subset \{1, \ldots, m\}}}~\widehat{f^1_S} \widehat{f^2_S} \widehat{f^3_S}\right| &\leqslant
        \underset{S \neq \emptyset}{\sum_{S \subset \{1, \ldots, m\}}}~|\widehat{f^2_S}| |\widehat{f^3_S}| |\widehat{f^1_S}| \leqslant
        \max_{S \neq \emptyset}{|\widehat{f^1_S}|} \underset{S \neq \emptyset}{\sum_{S \subset \{1, \ldots, m\}}}~|\widehat{f^2_S}| |\widehat{f^3_S}| \leqslant\\
        &\leqslant \max_{S \neq \emptyset}{|\widehat{f^1_S}|} \sqrt{Var(f^2) Var(f^3)} \leqslant
        \max_{S \neq \emptyset}{|\widehat{f^1_S}|} =
        \max_{i \in \{1, \ldots, m \}} {\widehat{f^1_{\{i\}}}} = \\
        &= \max_{i \in \{1, \ldots, m \}} P(\{x \in \{-1,+1\}^m :x  \text{ является }i\text{-критической для }f^1\}).
        \end{align*}
    \end{proof}

    \begin{lemma} Для любой клетки $i \in \{1, \ldots, m\}$ множества $M_n$ выполнено
    \label{rusineq2}
        \begin{align*}
            P(\{x \in &\{-1,+1\}^m :x \text{ является }i\text{-критической для }f^1\}) \leqslant \\
             \leqslant \min\{&P(\{x \in \{-1,+1\}^m :\text{для } x  \text{ существует  протекание из клетки }i\text{ до стороны }1\}),\\
             &P(\{x \in \{-1,+1\}^m :\text{для } x \text{ существует  протекание из клетки }i\text{ до стороны }4\})\}.
        \end{align*}
    \end{lemma}

    \begin{proof}
        В произвольной $i$-критической раскраске $x$ некоторую клетку стороны $1$ можно соединить c некоторой клеткой стороны $4$ цепочкой соседних по стороне клеток, каждая из которых имеет цвет $+1$. Но в этой цепочке клеток обязательно должна быть клетка $i$, так как при изменении знака в этой клетке протекания между сторонами нет. Значит, данную цепочку клеток можно разбить на две других цепочки, в одной из которых есть протекание из клетки $i$ до стороны $1$, а в другой --- протекание из клетки~$i$ до стороны~$4$. Отсюда получаем требуемое неравенство.
    \end{proof}

\subsection{Завершение доказательства}

    Пусть $i \in M_n$ --- произвольная клетка (см. рис. 2a). Обозначим через $r(i)$ максимальное расстояние от центра клетки $i$ до прямых, содержащих стороны $1$ и $4$ шестиугольника $P$. Обозначим через $P(i)$ прямоугольник $2r(i) \times 8r(i)$ с центром в центре клетки~$i$, большая сторона которого содержит сторону $1$ или $4$ шестиугольника~$P$ (см. рис. 2b). Далее без ограничения общности полагаем, что эта сторона прямоугольника содержит сторону $1$ шестиугольника.

    \begin{lemma}
    \label{rusrectangle}
        Шестиугольник $P$ целиком содержится в прямоугольнике $P(i)$.
    \end{lemma}

    \begin{proof}
        (См. рис. 2c). Введём такую прямоугольную систему координат с центром в центре клетки $i$, что ось $x$ параллельна прямой, содержащей сторону $1$. Тогда прямоугольник $2r(i) \times 8r(i)$ задаётся системой
        \begin{equation}
        \label{rusrect}
            \begin{cases}
                -4r(i) \leqslant x \leqslant 4r(i),\\
                -r(i) \leqslant y \leqslant r(i).
            \end{cases}
        \end{equation}
        "Большая диагональ" (см. рис. 2d) шестиугольника~$P$ равна $2$, поэтому шестиугольник~$P$ заведомо содержится внутри прямоугольника
        \begin{equation}
        \label{rusrect2}
            \begin{cases}
                -2 \leqslant x \leqslant 2,\\
                -r(i) \leqslant y \leqslant r(i).
            \end{cases}
        \end{equation}
        "Малая диагональ" (см. рис.~2d) шестиугольника~$P$ равна $\sqrt{3}$, поэтому $r(i)\geqslant\sqrt{3}/2$. Значит, если точка $(x, y)$ удовлетворяет неравенствам~(\ref{rusrect2}), то она также удовлетворяет неравенствам~(\ref{rusrect}).
    \end{proof}

    \begin{figure}[h]
     \centering
     \begin{subfigure}[t]{0.23\textwidth}
        \raisebox{-\height}{\includegraphics[width=\textwidth]{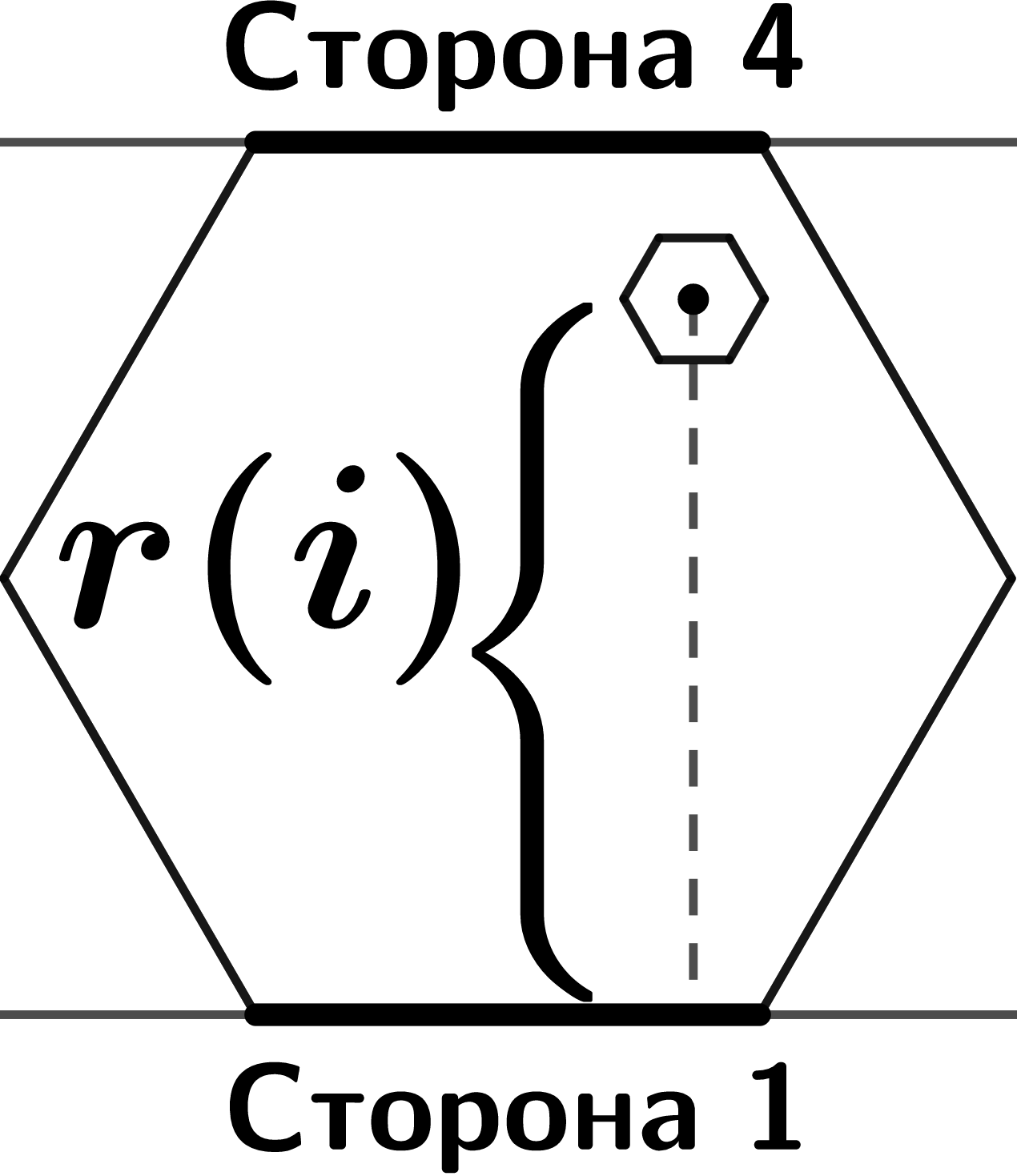}}
        \caption{}
     \end{subfigure}
    \hfill
     \begin{subfigure}[t]{0.30\textwidth}
        \raisebox{-\height}{\includegraphics[width=\textwidth]{2rx8r_2}}
        \caption{}
        \vspace{.3ex}
        \raisebox{-\height}{\includegraphics[width=\textwidth]{coordinates2.png}}
        \caption{}
     \end{subfigure}
    \hfill
     \begin{subfigure}[t]{0.40\textwidth}
        \raisebox{-\height}{\includegraphics[width=\textwidth]{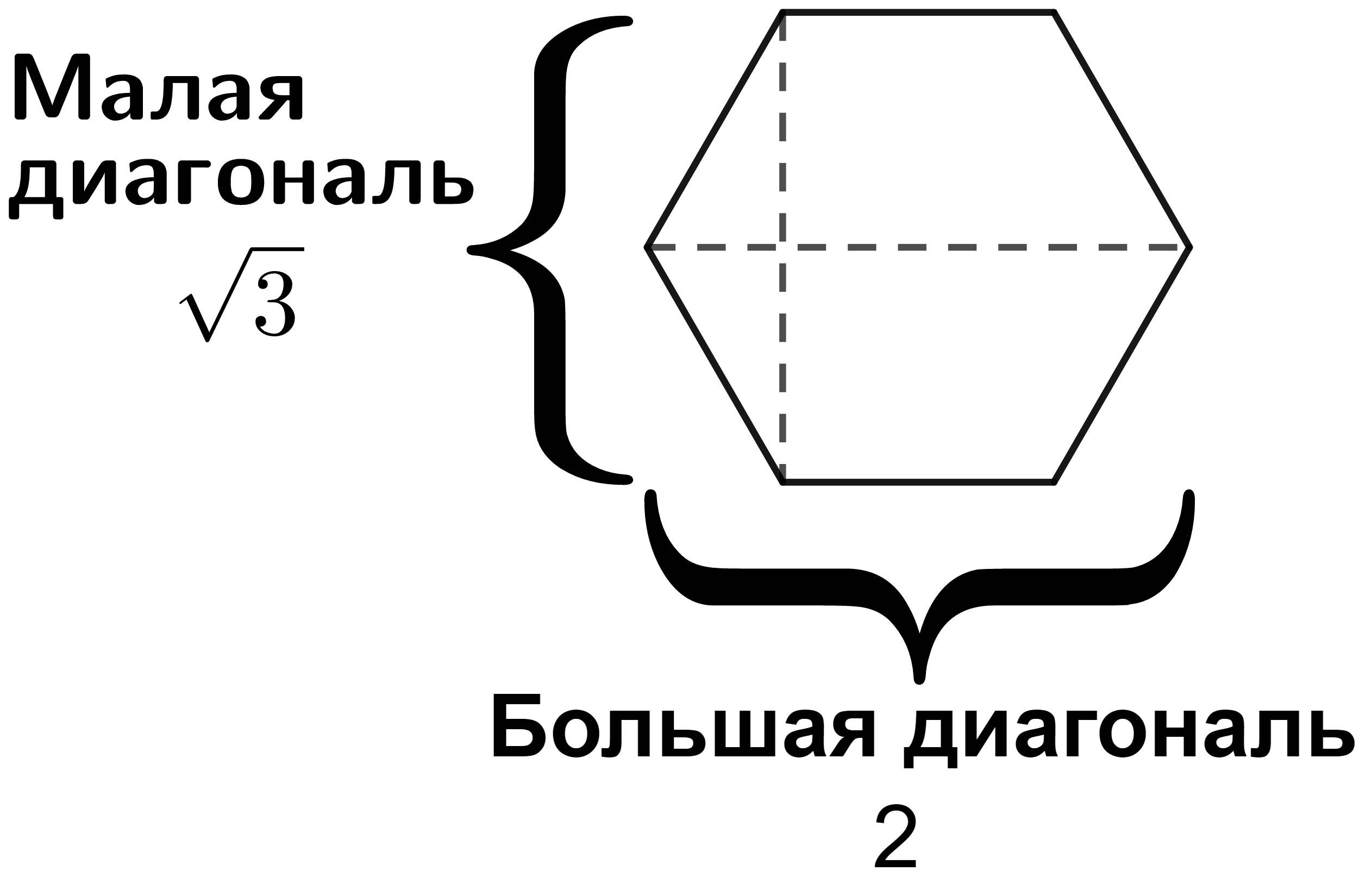}}
     \caption{}
     \end{subfigure}
     \caption{}
    \end{figure}

    Добавим к множеству $M_n$ все клетки решётки, которые содержатся в прямоугольнике $P(i)$. Получим новое множество $M'_n$. Будем обозначать раскраски множества $M'_n$ в 2 цвета $\pm 1$ как точки множества $\{-1,+1\}^{|M'_n|}$.

    \begin{lemma}
    \label{rusineq3}
        Выполнены неравенства:
        \begin{align*}
            &P(\{x \in \{-1,+1\}^m : \text{для } x \text{ существует протекание из клетки } i \text{ до стороны } 1\}) \leqslant\\
            &P(\{x \in \{-1,+1\}^{|M'_n|} : \text{для } x \text{ существует протекание из клетки } i \text{ до стороны } 1\}) \leqslant\\
            &P(\{x \in \{-1,+1\}^{|M'_n|} : \text{для } x \text{ существует протекание из клетки } i \text{ до границы } M'_n\}).
        \end{align*}
    \end{lemma}

    \begin{proof}
        По лемме \ref{rusrectangle} получаем, что $M_n \subset M'_n$. Первое неравенство следует из того, что если в раскраске $x \in \{-1,+1\}^m$ существует искомое протекание, то протекание существует во всех раскрасках $y \in \{-1,+1\}^{|M'_n|}$, которые совпадают с раскраской $x$ на множестве~$M_n$.

        Докажем второе неравенство. Если в какой-нибудь раскраске множества $M_n$ существует протекание из клетки $i$ до стороны $1$, то по лемме \ref{rusrectangle} в этой же раскраске существует протекание из клетки $i$ до границы множества~$M'_n$.
    \end{proof}

    \begin{lemma}
    \label{rusosnova}
        Выполнены неравенства
        \begin{align*}
            \min\{&P(\{x\in\{-1,+1\}^m : \text{для } x \text{ существует  протекание из клетки }i\text{ до стороны }1\}), \\
            &P(\{x\in\{-1,+1\}^m : \text{для } x \text{ существует  протекание из клетки }i\text{ до стороны }4\})\}           \leqslant \\
            P(\{&\text{в множестве } M_n \text{ существует протекание из клетки } O \text{ до границы}\}).
        \end{align*}
    \end{lemma}

    \begin{proof}
        Рассмотрим правильный шестиугольник с центром в центре клетки $i$, полученный из шестиугольника $P$ параллельным переносом. Из того, что $r(i) \geqslant \sqrt{3}/2$ (см. рис.~2d) следует, что построенный шестиугольник целиком содержится в прямоугольнике $P(i)$. С учётом этого из леммы \ref{rusineq3} получаем требуемое.
    \end{proof}

    \begin{proof}[Доказательство теоремы 1]
    Теорема 1 непосредственно получается из следствия \ref{russides refolmulation}, лемм \ref{rusineq}, \ref{rusineq2} и \ref{rusosnova} и теоремы \ref{ruskesten lite}.
    \end{proof}

\subsection{Аргументы в пользу истинности гипотезы \ref{rushyp2}}

    В завершение приведём неформальные аргументы в пользу истинности гипотезы \ref{rushyp2}. В вырожденном случае, при котором шестиугольник $P$ состоит из одной клетки, выполнено \\$\frac{P(A_{1,n} \cap A_{2,n} \cap A_{3,n})}{P(A_{1,n}) P(A_{2,n}) P(A_{3,n})}~=~1$. В примере 1 множество $M_n$ стало больше, и теперь $\frac{P(A_{1,n} \cap A_{2,n} \cap A_{3,n})}{P(A_{1,n}) P(A_{2,n}) P(A_{3,n})}~>~1$. Можно предположить, что при увеличении количества клеток в множестве $M_n$ искомое отношение тоже увеличивается. Следующий более сильный аргумент принадлежит К. Изъюрову.

    Рассмотрим следующую функцию от раскраски $x \in \{-1,+1\}^m$:
    \begin{equation*}
    \chi(x) = \chi(x_1, \ldots, x_m)=
    \begin{cases}
        \multirow{2}{*}{1,} &\text{если в раскраске $x$ существует протекание} \\
                            &\text{между центром и границей;}\\
        0, &\text{иначе.}
    \end{cases}
    \end{equation*}

    Сформулируем без доказательства лемму. (Доказательство аналогично доказательству леммы \ref{rusformulae for sides probability}.)
    \begin{lemma}
    \label{ruscenter fourier}
     $P(A_{1,n}) = P(A_{2,n}) = P(A_{3,n}) = \widehat{\chi_{\emptyset}}$ и $P(A_{1,n} \cap A_{2,n} \cap A_{3,n}) =~\displaystyle{\sum_{S \subset \{1, \ldots, m\}}~\widehat{\chi_S}^3}.$
    \end{lemma}

    \begin{corollary}
    $P(A_{1,n} \cap A_{2,n} \cap A_{3,n}) -  P(A_{1,n}) P(A_{2,n}) P(A_{3,n}) =
    \displaystyle{\underset{S \neq \emptyset}{\sum_{S \subset \{1, \ldots, m\}}} \widehat{\chi_S}^3}.$
    \end{corollary}
    Значит, гипотезы \ref{rushyp1}, \ref{rushyp2} утверждают, что
    $$\displaystyle{\underset{S \neq \emptyset}{\sum_{S \subset \{1, \ldots, m\}}} \widehat{\chi_S}^3} > 0\text{ и }\displaystyle{\lim_{n \to +\infty}}{\underset{S \neq \emptyset}{\sum_{S \subset \{1, \ldots, m\}}} \left(\frac{\widehat{\chi_S}}{\widehat{\chi_{\emptyset}}}\right)}^3 > 0.$$

    Последнее выражение содержит, в частности,
    $$
    \sum_{\text{$s$ соседняя с центром $O$}} \left(\frac{\widehat{\chi_S}}{\widehat{\chi_{\emptyset}}}\right)^3.
    $$ 
    Из этих 6 клеток рассмотрим какую-нибудь клетку~$i$. Коэффициент $\chi_{\{i\}}$ равен вероятности того, что случайная раскраска $i$-критическая. Легко убедиться, что условная вероятность $i$-критичности при условии протекания (т.е. $\chi_{\{i\}} / \chi_{\emptyset}$) весьма мала, но при этом положительна и не стремится к нулю при больших $n$. Так что эти 6 слагаемых дают вклад, который сам по себе не исчезает в пределе (и искомый предел может быть равен $0$, только если этот вклад магическим образом сократится со старшими коэффициентами Фурье-Уэлша, для чего не видно оснований). Поскольку вклад маленький, точность численных экспериментов не позволяет его увидеть.
\section{Численные эксперименты}
\label{ruscomputer}
    Для проверки гипотез \ref{rushyp1}-\ref{rushyp2} о протекании из центра была написана компьютерная программа (см. текст программы на C++ в \cite{rusprogramm}).

    Каждая сторона шестиугольника $P$ была выбрана параллельной некоторой стороне клетки решётки. Программа рассматривает $k$ случайных раскрасок многоугольника $M_n$, и в зависимости от введённого значения $p$ вычисляет либо количество раскрасок $T_1$, в которых жидкость $1$ протекает из центра до границы (при $p=1$), либо количество раскрасок $T_2$, в которых жидкости $1$ и $2$ протекают из центра до границы (при $p=2$), либо количество раскрасок $T_3$, в которых все три жидкости 1, 2 и 3 протекают из центра до границы (при $p=3$):\\

    \begin{center}
    \begin{tabular}{|c|c|}
        \hline
        Вход & Выход\\
        \hline
        $n$ & $T_p$\\
        $k$ & \\
        $p$ & \\
        \hline
    \end{tabular}
    \end{center}

    Результаты работы программы для некоторых $n$, $k$ и $p$ показаны в таблице \ref{rusprogram}. Дополнительно там приведено значение $P = (T_3/k)/(T_1/k)^3 \approx \frac{P(A_{1,n} \cap A_{2,n} \cap A_{3,n})}{P(A_{1,n}) P(A_{2,n}) P(A_{3,n})}$. Наличие значений $P < 1$ не позволяет опровергнуть гипотезу \ref{rushyp1}, так как может быть вызвано погрешностью в вычислении вероятности. При $n=500, k=1 000 000, p=1$ программа работает приблизительно 2 часа.

    \begin{table}[H]
    \caption{Результаты работы программы}
    \begin{center}
    \label{rusprogram}
    \begin{tabular}{|c|c|c|c|c|c|c|}
        \hline
        n & k & $T_1$ & $T_2$ & $T_3$ & $P$ & $P-1$\\
        \hline
        3& 10 000 000& 9 844 112& 9 690 354& 9 538 785& 1.000 & <0.001\\
        4& 10 000 000& 9 843 390& 9 691 475& 9 537 528& 1.000 & <0.001\\
        5& 10 000 000& 9 655 567& 9 327 500& 9 009 425 & 1.001 & 0.001\\
        6& 10 000 000& 9 575 758& 9 166 785 & 8 779 321& 1.000 & <0.001\\
        7& 10 000 000& 9 478 577& 8 985 300& 8 522 311& 1.001 & 0.001\\
        8& 10 000 000& 9 368 765& 8 778 158& 8 229 285& 1.001 & 0.001\\
        9& 10 000 000& 9 285 683& 8 630 914& 8 016 759& 1.001 & 0.001\\
        10& 10 000 000& 9 211 830& 8 475 622& 7 813 657& 1.000 & <0.001\\
        15& 10 000 000& 8 898 951& 7 893 907& 7 030 482& 0.998 & -0.002\\
        20& 10 000 000& 8 657 067& 7 466 236& 6 462 095& 0.996& -0.004\\
        25& 10 000 000& 8 461 617& 7 166 402& 6 016 836& 0.993& -0.007\\
        50& 1 000 000& 787 626& 621 883& 490 967& 1.005& 0.005\\
        100& 1 000 000& 733 994& 539 606& 395 491& 1.000 & <0.001\\
        150& 1 000 000& 707 948& 500 289& 354 320& 0.999 & -0.001\\
        200& 1 000 000& 681 800& 466 133& 318 058& 1.004 & 0.004\\
        300& 1 000 000& 660 490& 429 462& 283 512& 0.984 & -0.016\\
        400& 1 000 000& 633 621& 402 712& 259 337& 1.019 & 0.019\\
        500& 1 000 000& 621 187& 386 527& 242 821& 1.013& 0.013\\
        1500& 1 000 000& 555 244& 299 740& 166 587& 0.973 & -0.027\\
        \hline
    \end{tabular}
    \end{center}
    \end{table}

    Опишем алгоритм, определяющий, протекает ли жидкость $1$ из центра $O$ до границы в заданной раскраске. Поскольку цвет центра не имеет значения, то мы будем считать, что этот цвет $0$ (см. примеры в таблице \ref{rusexamples}).

    Рассмотрим некоторую раскраску. Посадим в центр жука, способного переходить в соседние клетки цвета $0$ или $1$. Мы хотим узнать, может ли жук достичь граничной клетки.

    \setlength{\parindent}{0pt}

    \textbf{Шаг~1.} Если жук уже в граничной клетке, то переходим в КОНЕЦ1. Иначе переходим к Шагу~2.

    \textbf{Шаг~2.} Если цвет клетки под жуком $0$ или $1$, то двигаем жука вниз, и переходим к Шагу~1.
    Если цвет клетки под жуком $2$ или $3$, то переходим к Шагу~3.

    \textbf{Шаг~3.} (Построение стены; см. таблицу \ref{rusexamples}.)

    Построим ломаную (впоследствии называемую \textit{стеной}), идущую по сторонам клеток (называемых \textit{сегментами стены}) следующим образом. Первым звеном является нижняя сторона клетки, в которой жук находится в данный момент. Второе звено примыкает к левому концу первого. Все клетки по одну сторону стены имеют цвет $0$ или $1$, а все клетки по другую сторону стены --- цвет $2$ или $3$. Наконец, либо последний сегмент стены лежит между двумя граничными клетками, либо все сегменты стены окружают какую-нибудь часть клеток в $M_n$. В первом случае переходим в КОНЕЦ1, во втором случае переходим к Шагу~4.

    (Легко видеть, что стена однозначно определяется указанными условиями.)

    \textbf{Шаг~4.} Проведём луч из центра множества $M_n$ вниз. Если луч пересекает нечётное число сегментов стены (которые были построены в Шаге~3), то переходим в КОНЕЦ2. В противном случае рассмотрим самый нижний сегмент стены, пересекающийся с лучом. Передвинем жука в клетку, находящуюся под эти сегментом, и переходим к Шагу~1.

    \textbf{КОНЕЦ1.} Жидкость $1$ протекает из центра до границы.

    \textbf{КОНЕЦ2.} Жидкость $1$ не протекает из центра до границы.

    Таблица на следующей странице показывает, как работает алгоритм. Во всех трёх примерах $n=7$.

    \begin{minipage}[t]{\textwidth}
    \centering
    \captionof{table}{Алгоритм проверки наличия протекания}
    \label{rusexamples}
    \begin{tabular}{| p{4.1cm} | p{4.1cm} | p{4.1cm} | p{4.1cm} |}
        \hline
        \multicolumn{4}{|c|}{\includegraphics[trim=0 60 0 -2, width=0.75\textwidth]{colors.png}}\\
        \hline
        \multicolumn{4}{|c|}{Пример 1} \\
        \hline
            \includegraphics[trim=0 0 0 -1,width=0.24\textwidth]{coloring11.png}
            & \includegraphics[width=0.24\textwidth]{coloring12.png}
            & \includegraphics[width=0.24\textwidth]{coloring13.png}
            & \includegraphics[width=0.24\textwidth]{coloring14.png}\\
        \hline
        Изначально жук сидит в центре.&
        Жук двигается вниз до тех пор, пока не упрётся в синюю клетку. Ставим сегмент стены под жуком.&
        Строим стену до тех пор, пока не поставим сегмент стены между двумя соседними граничными клетками.&
        Идя вдоль стены, жук может достичь границы. \\
        \hline
        \multicolumn{4}{|c|}{Пример 2} \\
        \hline
            \includegraphics[trim=0 0 0 -2, width=0.24\textwidth]{coloring21.png}
            & \includegraphics[width=0.24\textwidth]{coloring22.png}
            & \includegraphics[width=0.24\textwidth]{coloring23.png}
            & \includegraphics[width=0.24\textwidth]{coloring24.png}\\
        \hline
        Изначально жук сидит в центре.&
        Жук двигается вниз до тех пор, пока не упрётся в красную клетку. Ставим сегмент стены под жуком.&
        Строим стену, пока она не окружит некоторую часть множества $M_7$.&
        Проводим луч из центра вниз. Этот луч пересекает только один сегмент стены, следовательно, жук не может достичь границы.\\
        \hline
        \multicolumn{4}{|c|}{Пример 3} \\
        \hline
            \includegraphics[trim=0 0 0 -2,width=0.24\textwidth]{coloring31.png}
            & \includegraphics[width=0.24\textwidth]{coloring32.png}
            & \includegraphics[width=0.24\textwidth]{coloring33.png}
            & \includegraphics[width=0.24\textwidth]{coloring34.png}\\
        \hline
        Изначально жук сидит в центре. Он не может двигаться вниз, поэтому мы устанавливаем сегмент стены под ним.&
        Строим стену, пока она не окружит некоторую часть множества $M_7$.&
        Проводим луч из центра вниз. Этот луч пересекает два сегмента стены.&
        Жук обходит стену. Затем он двигается вниз и достигает границы.\\
        \hline
    \end{tabular}
    \end{minipage}

\section{Благодарности}
    Автор выражает благодарность своему научному руководителю М.~Скопенкову за постановку задач и постоянное внимание к данной работе. Также автор благодарит К.~Изъюрова, А.~Магазинова и М.~Христофорова, прочитавших данную работу и приславших ценные замечания. Автор благодарит К.~Изъюрова и А.~Магазинова за то, что рассказали доказательство теоремы~\ref{russides}.
    {}